\documentclass[11pt, reqno]{amsart}
\usepackage{amsmath, amssymb, amsthm, graphicx,marvosym,mathrsfs}
\usepackage{soul,bbm,hyperref}
\usepackage{cancel}
\usepackage{cite}
\usepackage[nobysame]{amsrefs}[11pts]
\usepackage{bm}
\usepackage{color}
\allowdisplaybreaks
\newtheorem{theorem}{Theorem}[section]

\newtheorem{lemma}{Lemma}[section]

\newtheorem{proposition}{Proposition}[section]

\numberwithin{equation}{section}
\numberwithin{figure}{section}

\makeatletter \@addtoreset{equation}{section} \makeatother

\newcommand{\R}{\mathbb{R}}
\newcommand{\T}{\mathbb{T}}

\newcommand{\eps}{\mathbb{\varepsilon}}
\newcommand{\al}{\alpha}
\newcommand{\dive}{\mathrm{div}}

\newcommand{\pt}{\partial_t}
\newcommand{\pa}{\partial}

\newcommand{\pie}{\prime}
\newcommand{\cubic}{C\norm{W}_s\norm{U}_s^2}
\newcommand{\cubict}{C\int_0^T\norm{W(t)}_s\norm{U(t)}_s^2\mathrm{d}t}
\newcommand{\D}{\nabla}

\newcommand{\N}{\mathcal{N}}
\newcommand{\U}{\mathcal{U}}
\newcommand{\F}{\mathcal{F}}
\newcommand{\G}{\mathcal{G}}
\newcommand{\V}{\mathcal{V}}
\newcommand{\W}{\mathcal{W}}
\newcommand{\E}{\mathcal{E}}
\newcommand{\mD}{\mathcal{D}}
\newcommand{\norm}[1]{{\left\vert\kern-0.25ex\left\vert\kern-0.25ex\left\vert #1\right\vert\kern-0.25ex\right\vert\kern-0.25ex\right\vert}}

\begin{document}
	\title[infinity-ion-mass limits for Euler-Poisson system]{\bf Approximations of Euler-Maxwell systems by drift-diffusion equations through zero-relaxation limits near non-constant equilibrium}

	\author{Rui Jin}
	\address[R. Jin]{School of Mathematical Sciences, 
	Shanghai Jiao Tong University, Shanghai 200240, P. R. China}
	\email{\tt jinrui@sjtu.edu.cn}
	
	\author{Yachun Li}
	\address[Y. Li]{School of Mathematical Sciences, CMA-Shanghai, MOE-LSC, and SHL-MAC, Shanghai Jiao Tong University, Shanghai 200240, P. R. China}
	\email{\tt ycli@sjtu.edu.cn}
	
    \author{Liang Zhao}
    \address[L. Zhao]{Mathematical Modelling \& Data Analytics Center, Oxford Suzhou Centre for Advanced Research, Suzhou 215123, P.R. China}
    \email{\tt liang.zhao@oxford-oscar.cn}
	
	\date{\today}
	\maketitle \markboth{Approximations of the Euler-Maxwell near non-constant equilibrium}{R. Jin, Y. Li and L. Zhao}
	
\begin{abstract}
    Due to extreme difficulties in numerical simulations of Euler-Maxwell equations, which are caused by the highly complicated structures of the equations, this paper concerns the simplification of Euler-Maxwell system through the zero-relaxation limit towards the drift-diffusion equations with non-constant doping functions. We carry out the global-in-time convergence analysis by establishing uniform estimates of solutions near non-constant equilibrium regarding the relaxation parameter and passing to the limit by using classical compactness arguments. Furthermore, stream function methods are carefully generalized to the non-constant equilibrium case, with which as well as the anti-symmetric structure of the error system and an induction argument, we establish global-in-time error estimates between smooth solutions to the Euler-Maxwell system and those to drift-diffusion system, which are bounded by some power of relaxation parameter. 
\end{abstract}

\section{Introduction}
A plasma is a collection of moving electrons and ions.  In modern plasma industries, the numerical simulations of particle motions in plasma environments become more and more important. Mathematically speaking, the movements of electrons and ions in plasma can both be regarded as fluid motions, which can be modelled by Euler and Maxwell equations coupled through Lorentz forces. These equations are known as the two-fluid Euler-Maxwell system. Nevertheless, due to the complex mathematical structure of this system, simplifications should be made if one needs to carry out numerical simulations. Since electrons move faster, ions are often regarded as non-moving and perform as the background density. Consequently, on the scale of electrons, the equations for ions can be neglected (see \cite{Zhao2021}). Thus, the simplified one-fluid Euler-Maxwell system for electrons can be written into the form \cites{Besse2004,Chen1984,Chen2000}:
\begin{equation}\label{EM1}
	\begin{cases}
		{\pa _{t'}}n + \dive(nu) = 0,\\
		{\pa _{t'}}(nu) + \dive(nu \otimes u) + \D P(n) =  - n (E + u \times B) - \dfrac{nu}{\eps},\\
		{\pa _{t'}}E - \D  \times B = nu,\quad \dive E = b(x) - n ,\\
		{\pa _{t'}}B + \D  \times E = 0,\quad \dive B = 0,
	\end{cases}
\end{equation}
with $(t',x)\in\R^+\times \mathbb{K}^3$. The initial data are given by
\[
t' = 0:(n ,u,E,B) = (n_0 ,u_0 ,E_0 ,B_0)(x),\quad x \in \mathbb{K}^3.
\]
Here $x=(x_1,x_2,x_3)^\top$ and $t'>0$ are the space and usual time variables respectively, $\mathbb{K}=\mathbb{R}$ for Cauchy problems and $\mathbb{K}=\T$ for periodic problems with $\T^3$ a torus in $\mathbb{R}^3$. The unknowns are $n$, $u=(u_1,u_2,u_3)^\top$, $E=(E_1,E_2,E_3)^\top$ and $B=(B_1,B_2,B_3)^\top$, which denote electron density, electron velocity, electric field and magnetic field, respectively. They are all functions of $(t',x)$. The small parameter $\eps \in (0,1]$ denotes the relaxation time. The pressure function $P(n)$ is supposed to be sufficiently smooth and strictly increasing for all $n>0$. The given function $b(x)$ is the doping profile, which describes the distribution of background densities. We assume that there exists a positive constant $b_1>0$, such that
\begin{equation}\label{assb}
	b(x)\in L^{\infty}(\mathbb{K}^3), \quad b(x)\ge b_1>0\quad\text{and}\,\,\quad \nabla b\in H^{q'}(\mathbb{K}^3), \quad \text{with}\,\,q'\geq 3.
\end{equation}

However, due to the hyperbolic structure of \eqref{EM1}, the numerical simulations of which  still
face great challenges, in which high computing power, special algorithms and strong techniques are needed, especially for multi-dimensional cases (see Ref.~\cite{Degond2012}). In this paper, we wish to further simplify \eqref{EM1} under the zero-relaxation limit $\eps\to 0$. The  limit system is the classical drift-diffusion equations, which is parabolic-elliptic type and easier to carry out numerical simulations. System \eqref{EM1} fills in the framework of the famous Jin-Xin model \cite{Jin1995} and we refer readers to
Whitham,\cite{Whitham1974} Liu \cite{Liu1987} Serre \cite{Serre2000} and references cited therein for various results for relaxation corrections for conservation laws. The limit process can be described formally as follows. When the slow time scale $t=\eps t'$ is introduced and the following parabolic scaling is made,
$$(n^{\eps},E^\eps,B^\eps)(t,x)=(n,E,B)\left(t/\eps,x\right),\quad \eps u^{\eps}(t,x)=u\left(t/\eps,x\right),$$
the Euler-Maxwell system \eqref{EM1} becomes
\begin{equation}\label{EM2}
	\left\{ \begin{array}{l}
		{\pt}n^{\eps} + \dive(n^{\eps}u^{\eps}) = 0,\\
		\eps^2 \pt (n^\eps {u^{\eps}})+\eps^2 \dive(n^\eps u^\eps\otimes u^\eps)+\D{P(n^{\eps})}=-n^\eps(E^{\eps} +\eps u^{\eps} \times B^{\eps}) - n^\eps u^{\eps},\\
		\eps {\pt}E^{\eps} -\D\times B^{\eps} = \eps n^{\eps}u^{\eps},\quad \dive E^{\eps} = b(x) - n^{\eps} ,\\
		\eps{\pt}B^{\eps} + \D  \times E^{\eps} = 0,\quad \dive B^{\eps} = 0,\\
	\end{array} \right.
\end{equation}
with the initial conditions
\begin{equation}\label{EM2initial}
	t= 0:\left(n^{\eps} , u^{\eps},E^{\eps},B^{\eps}\right):= \left(n_0^\eps,u_0^\eps ,E_0^\eps ,B_0^\eps \right)(x)=\left(n_0,u_0/\eps ,E_0 ,B_0 \right)(x).
\end{equation}
When $\eps\rightarrow0$, if denoting the limits of $(n^\eps, u^\eps, E^\eps, B^\eps)$ as $(\bar{n},\bar{u},\bar{E},\bar{B})$, the formal limit system is of the form:
\begin{equation}\label{EM-limit}
	\left\{ \begin{array}{l}
		{\pa _{t}}\bar{n} + \dive(\bar{n}\bar{u}) = 0,\\
		\D{P(\bar{n})}=-\bar{n}\bar{E}-\bar{n}\bar{u},\\
		\D\times\bar{B}=0,\;\dive\bar{E} = b(x) - \bar{n},\\
		\D \times \bar{E} = 0,\;\dive\bar{B}=0.
	\end{array} \right.
\end{equation}
Since $\D\times\bar{E}=0$, there exists a unique potential function $\bar{\phi}$ satisfying $ \bar{E}=\D\bar{\phi}$ with
\begin{equation}\label{conditionphi}
	m_{\bar{\phi}}(t):=\int_{\T^3}\bar{\phi}(t,x)dx=0, \,\text{for}\,x\in\T^3,\,\, \text{or}\,\,  \lim_{|x|\to\infty}\bar{\phi}(t,x)=0,\,\,\text{for}\,\,x\in\mathbb{R}^3.
\end{equation}
Consequently, \eqref{EM-limit} can be rewritten into the classical drift-diffusion model
\begin{equation}\label{EM-drift-diffusion}
	\left\{\begin{array}{l}
		\pt\bar{n}-\Delta P(\bar{n})-\dive(\bar{n}\D\bar{\phi})=0,\\
		\Delta\bar{\phi}=b(x)-\bar{n},
	\end{array}\right.
\end{equation}
with an additional relation:
\begin{equation}\label{baru}
	\bar{u}=-\frac{1}{\bar{n}}\D\left(P(\bar{n})+\bar{\phi }\right).
\end{equation}

To prove rigorously that the simplification is valid globally-in-time in mathematics, one needs to first establish the global-in-time well-posedness of classical solutions to \eqref{EM1}. For $n>0$, system \eqref{EM1} can be regarded as a first-order symmetrizable hyperbolic system, hence the local-in-time existence and uniqueness of smooth solutions can be obtained by standard theories (see \cites{Kato1975,Lax1973,Majda1984}). It is well known that smooth solutions of hyperbolic systems usually exist locally-in-time and singularities may appear in finite time. However, the dissipative structure of the system may prevent the formation of singularities and leads to global-in-time existence of smooth solutions in a neighbourhood of an equilibrium state $W_e$. We refer readers to  Peng-Wang-Gu,\cite{Peng2011}  Xu \cite{Xu2011} and Ueda-Wang-Kawashima\cite{Ueda2012} for $W_e$ being a constant vector (or in other words $b(x)$ is a constant). We also refer to Germain-Masmoudi\cite{Germain2014} and Guo-Ionescu-Pausader\cite{Guo2016} for global existence of smooth solutions without the velocity dissipation term but with generalized irrotationality constraint $B=\D\times u$. However, for physical interest, the case of constant $W_e$ has many limitations. Generally, in the case when $W_e$ is not a constant (in other words $b(x)$ depends on $x$), but a stationary solution to \eqref{EM1} in which the velocity is zero, the global well-posedness theories near $W_e$ for \eqref{EM1} become more complicated. Let $W_e=(n_e,0,E_e,B_e)$ be the equilibrium satisfying
\begin{equation}\label{PHT}
	\begin{cases}
		\D P(n_e)=-n_eE_e, \\
		\D \times B_e=0, \quad \dive E_e=b(x)-n_e, \\
		\D \times E_e=0, \quad \dive B_e=0.
	\end{cases}
\end{equation}
We learn from the above equations that $B_e$ is a constant vector, and $n_e$ satisfies
\begin{equation}\label{1.4}
	-\Delta {h(n_e)}+n_e=b(x),\quad x\in\mathbb{K}^3,
\end{equation}
where $h$ is the enthalpy function defined as
\[
h'(n)=\frac{P'(n)}{n}.
\]
Since $h^\prime>0$, \eqref{1.4} is elliptic. For the global well-posedness of solutions to \eqref{EM1} near non-constant $W_e$, inspired by the ideas in Guo-Strass\cite{Guo2005}, Peng\cite{Peng2015} combined an anti-symmetric matrix technique and an induction argument to obtain global smooth solutions to Euler-Maxwell system. We also refer to Feng-Peng-Wang\cite{Feng2015} and Liu-Peng\cite{Liu2017} for stability problems for two-fluid models or non-isentropic ones. It is noted that these global well-posedness results are not uniform regarding the relaxation parameter $\eps$. For the zero-relaxation limit problem, the local-in-time and the global-in-time convergence of \eqref{EM1} as $\eps\to 0$ were obtained by Hajjej-Peng\cite{Hajjej2012} and Peng-Wang-Gu, \cite{Peng2011} respectively. In the local-in-time convergence result, the convergence rate was clearly shown and it depends on the local existence time. For global convergence, Li-Peng-Zhao\cite{Li2021} established the global-in-time error estimate between \eqref{EM1} and \eqref{EM-drift-diffusion}. All these global-in-time convergence analysis were carried out near constant equilibrium states. For non-constant $W_e$, no result has been reached so far, for either  global convergence or global convergence rates. For other related problems, we refer to \cites{Ali2000,Jungel1999,Peng2008,XuJiang2008} and references cited therein.

As simplification of Euler-Maxwell system, Euler-Poisson system is also an important model in plasma and semiconductors. The approximation mentioned above is also valid for Euler-Poisson system. The 3-D periodic problem for Euler-Poisson system can be written as (see \cites{Chen1984,1990Semiconductor})
\begin{equation}\label{EPorigin}
	\begin{cases}
		\partial_{t^{\prime}} n+\dive{(n u)}=0, \\
		\partial_{t^{\prime}}(n u)+\operatorname{div}(nu \otimes u)+\nabla P(n)=-n \nabla \phi-\dfrac{n u}{\varepsilon}, \\
		\Delta \phi=b(x)-n, \qquad m_{\phi}(t)=0,\\
			t^{\prime}=0: \quad(n, u)=\left(n_{0}, u_{0}\right)(x).
	\end{cases}
\end{equation}
Similarly, we introduce the slow time $t=\eps t^\pie$ and the following parabolic scaling,
\[  n^\eps(t,x)=n\left(t/\eps,x\right), \quad
\eps u^\eps(t,x)=\displaystyle u\left(t/\eps,x\right),
\quad \phi^\eps(t,x)=\phi\left(t/\eps,x\right),   \]
then $(n^\eps,u^\eps,\phi^\eps)$ satisfies the following periodic problem
\begin{equation}
	\label{EPmain}
	\begin{cases}
		\pt n^\eps+\dive ( n^\eps u^\eps)=0,\\
		\eps^2 \pt (n^\eps{u^{\eps}})+\eps^2 \dive(n^\eps u^\eps\otimes u^\eps)+\D{P(n^{\eps})}=-n^\eps \nabla \phi^\eps-n^\eps u^\eps,\\
		\Delta \phi^\eps=b(x)-n^\eps, \\
		t=0: \quad (n^\eps,u^\eps)(0,x):=(n_0^\eps(x),{u_0^\eps(x)})=(n_0(x),u_0(x)/\eps).
	\end{cases}
\end{equation}
By substituting $E_e$ and $\bar{E}$ with $\nabla\phi_e$ and $\nabla\bar{\phi}$ respectively, it holds that  $(n_e,0,\phi_e)$ and $(\bar{n},\bar{u},\bar{\phi})$ also satisfy \eqref{PHT}-\eqref{1.4} and \eqref{EM-drift-diffusion}-\eqref{baru}, respectively. For global well-posedness of solutions near constant or non-constant equilibrium states when $\eps=1$, we refer readers to Guo \cite{Guo1998}, Hsiao-Markowich-Wang \cite{Hsiao2003}, Fang-Xu-Zhang \cite{Fang2007}, Guo-Strauss \cite{Guo2005}, Huang-Mei-Wang \cite{Huang2011}, Germain-Masmoudi-Pausader \cite{Germain2013} and references cited therein. For the zero-relaxation limit problem for \eqref{EPorigin}, we refer to \cites{Lattanzio2000,yong2004diffusive} for the local-in-time convergence and Peng \cite{Peng2015b} for the global-in-time one, respectively. Similarly, for the case of non-constant equilibrium, no result has been reached so far, for either global-in-time convergence or global-in-time convergence rates.

The aim of this paper is to study the zero-relaxation limit $\eps\to 0$ and its global-in-time error estimates for Euler-Maxwell system \eqref{EM1} or Euler-Poisson system \eqref{EPorigin} near the general non-constant equilibrium states $W_e$, which satisfies \eqref{PHT}. The main difficulty appears in establishing the global-in-time error estimates. Usually, they are obtained by energy methods applied for the error system, which is the difference between original system \eqref{EM1} or \eqref{EPorigin} and the limit equations \eqref{EM-drift-diffusion}. However, the error system in our case shows neither hyperbolicity nor parabolicity. This makes it unclear the preservation of symmetrizable hyperbolic structure and the strictly convex entropy of the error system, which we are unable to use to close the estimates. To overcome these difficulties, stream function techniques should be applied. We begin with a review for this method. For a conservative equation
\begin{equation}
	\label{conserv}
	\pt z+\dive w =0,
\end{equation}
we call $\varphi$ a stream function associated to this equation if $\varphi$ satisfies
\begin{equation}\label{streamgeneral}
		\pt \varphi=w+K,\quad \dive \varphi =-z,
\end{equation}
where $K$ is some divergence-free terms. The key idea of this method is to take the inner product of $\varphi$ with terms such as $\D z$ to give an dissipative estimate for $z$. Apparently, the stream function is not unique.  The choice of $K$ is accurate and highly relies on the structure of the system, especially for multi-dimensional cases. In 2002, Junca-Rascle\cite{Junca2002} pioneeringly used this technique to establish the global-in-time $L^2$ error estimates between the 1D Euler equations and the heat equation with general pressure law. Inspired by their idea, there are several subsequent works in which stream function techniques are used to handle relaxation-type limits in Euler-type equations, see \cites{Goudon2013,Li2021,Zhao2021a} for other relevant studies. However, in our case, the non-constant equilibrium state brings additional difficulties compared with constant ones in that $\D n^\eps\neq \D(n^\eps-n_e)$, which yields that $\D n^\eps $ is not a small quantity. Moreover, $\D n_e$ does not depend on $t$ and thus can not provide any integrabilities with respect to $t$. These difficulties make it hard to treat the terms containing $\D n_e$ generated by integration by parts and thus the classical stream function method is not valid.

Our strategies are as follows. We notice that the process of establishing uniform estimates with respect to the small parameter $\eps$ can be regarded as the estimates of the error system between the original system \eqref{EM2} and the equilibrium one \eqref{PHT}. This inspires us to reformulate the error system between \eqref{EM2} and \eqref{EM-drift-diffusion} into an anti-symmetric form, based on which the $L^2$-estimate of $u^\eps-\bar{u}$ is obtained. Furthermore, we generalize the  stream function technique to cases of non-constant equilibrium states, together with which  the induction argument enable us to obtain estimates for $n^\eps-\bar{n}$, $E^\eps-\bar{E}$ and $B^\eps-B_e$. In summary, we need a precise combination of three symmetric structures of Euler-Maxwell systems, i.e., symmetrizable hyperbolic structure, anti-symmetric structure and the symmetric structure of the zeroth order term. It is worth emphasizing that this is highly non-trivial and very different from the classical energy estimates or the case treated in \cite{Li2021}.

This paper is organized as follows. \S 2 introduces preliminaries and main results. \S 3 concerns the global convergence analysis from Euler-Maxwell system to drift-diffusion system. We first establish the uniform estimates of smooth solutions near the non-constant equilibrium states with respect to the small parameter $\eps$, and using the theories of compactness to obtain the global-in-time convergence. \S 4 is devoted to the global-in-time error estimates between smooth solutions to \eqref{EM2} and \eqref{PHT}. The application of our methods to Euler-Poisson system is in \S 5.



\section{Preliminaries and main results}

\subsection{Notations and inequalities.}
For later purpose, we introduce the following notations. We denote  $\|\cdot\|$, $\|\cdot\|_{\infty}$ and $\|\cdot\|_{s}$ the norms of the usual Sobolev spaces $L^{2} :=L^{2}\left(\mathbb{K}^3\right)$, $L^{\infty} :=L^{\infty}\left(\mathbb{K}^3\right)$ and $H^{s} :=H^{s}\left(\mathbb{K}^3\right)$, respectively. $\left<\cdot,\cdot\right>$ stands for the inner product in $L^{2}$. For multi-indices $\al=\left(\al_{1}, \al_{2}, \al_{3}\right) \in \mathbb{N}^{3}$,  we denote
\[
\pa_x^{\al}=\frac{\pa^{|\al|}}{\pa x_{1}^{\al_{1}} \pa x_{2}^{\al_{2}} \pa x_{3}^{\al_{3}}},\quad |\al|=\al_{1}+\al_{2}+\al_{3}.
\]
For any fixed $T>0$, let
$$B_{s, T}(\mathbb{K}^3)=\bigcap_{l=0}^{s} \mathcal{C}^{l}\left([0, T] ; H^{s-l}(\mathbb{K}^3)\right).$$
For all $ t\in[0,T]$, we define the norm
$$\norm{f(t,\cdot)}_{s}^2=\sum_{l+|\al| \le s}\|\pa_{t}^{l} \pa_{x}^{\al} f(t, \cdot)\|^{2},\quad\forall \;f\in B_{s, T}(\mathbb{K}^3).$$

We first state the existence results for equilibrium system \eqref{1.4}, which can be obtained  by a minimization method or a classical Schauder fixed point theorem for $\mathbb{K}=\T$ and by the variational method for Cauchy problem $\mathbb{K}=\mathbb{R}$.

\begin{proposition}(Existence of equilibrium solutions, see \cite{Liu2019})\label{extcequil}
	Let $q\geq 4$ be an integer. Suppose the conditions on $b(x)$ introduced in \eqref{assb} hold. Then there exists a positive constant $\underline{n}>0$, such that \eqref{1.4} admits a unique classical solution $n$ satisfying $n-b\in H^{q-1}(\mathbb{K}^3)$ with $n\geq \underline{n}>0$. In particular, $n\in W^{q-2,\infty}(\mathbb{K}^3)$.
\end{proposition}

The next lemma concerns the estimates for product functions.
\begin{lemma} Let positive integers $ k\leq s$ and multi-indices $1\leq |\al|\leq s$ with $1\leq |\al|+k\leq s$. Let $u,v\in B_{s, T}(\mathbb{K}^3)$. For simplicity, we denote $I:=k+|\al|$. Then,
	\begin{eqnarray}\label{Moser00}
		\|\pt^k\pa_x^\al(uv)-u\pt^k\pa_x^\al v\|&\leq& C\|\D u\|_{s-1}\|\pt^k v\|_{|\al|-1}+C\norm{\pt u}_{s-1}\norm{v}_{I-1}, \\
\label{Moser11}
		\|\pt^k\pa_x^\al (uv)\|&\leq& C\norm{u}_s\norm{v}_{I}.
	\end{eqnarray}
	In addition, when $|\al|=0$, for integers $1\leq l\leq s$, one has
	\begin{equation}\label{Mosertt}
		\|\pt^l(uv)-u\pt^l v\|\leq C\norm{\pt u}_{s-1}\norm{v}_{l-1}, \quad \|\pt^l(uv)\|\leq C\norm{u}_{s}\norm{v}_{l}.
	\end{equation}
\end{lemma}
\begin{proof} The proof is mainly based on the following fact that for $f,g\in H^1$, it holds
	\[
	\|fg\|\leq C\|f\|_{L^6}\|g\|_{L^3}\leq C\|f\|_1\|g\|_1^{1/2}\|g\|^{1/2}\leq C\|f\|_1\|g\|_1.
	\]

	We first prove \eqref{Moser00}. Notice  $\pt^k\pa_x^\al(uv)-u\pt^k\pa_x^\al v$ is composed of terms as:
	\[
	\pt^l\pa_x^\beta u\,\pt^m\pa_x^\gamma v, \quad l+m=k,\,\, |\beta|+|\gamma|=|\al|,\,\, l+|\beta|\geq 1.
	\]
	We treat different cases as follows.
	
	\vspace{3mm}
	
	\noindent \underline{Case A: $k=0$ and $|\al|=I$}. This is the case when no time derivatives are applied. Then \eqref{Moser00} reduces to classical Moser-type calculus inequalities (see \cites{Majda1984,Peng2015}).
	
	\vspace{3mm}
	
	\noindent \underline{Case B: $k\geq 1$, $1\leq |\al|\leq I-1$ and $l=0$}. One has $m=k$ and $|\beta|\geq 1$. Then it holds
	
	\vspace{2mm}
	
	For $|\beta|=1$, one has $|\gamma|=|\al|-1$. Then
	\[
	\|\pa_x^\beta u\,\pt^k\pa_x^\gamma v\|\leq C\|\D u\|_\infty\|\pt^k\pa_x^\gamma v\|\leq C\|\D u\|_{s-1}\|\pt^k v\|_{|\al|-1}.
	\]
	
	For $2\leq |\beta|\leq s-1$, one has $|\gamma|\leq |\al|-2$, then
	\[
	\|\pa_x^\beta u\,\pt^k\pa_x^\gamma v\|\leq C\|\pa_x^\beta u\|_1\|\pt^k\pa_x^\gamma v\|_1\leq C\|\D u\|_{s-1}\|\pt^k v\|_{|\al|-1}.
	\]
	
	For $|\beta|=s$, one has necessarily $|\al|=s, k=0, \gamma=0$. Then
	\[
	\|\pa_x^\beta u\,v\|\leq C\|\pa_x^\beta u\|\|v\|_\infty\leq C\|\D u\|_{s-1}\|v\|_{s-1}=C\|\D u\|_{s-1}\|\pt^k v\|_{|\al|-1}.
	\]
	
	\vspace{3mm}
	
	\noindent \underline{Case C: $k\geq 1$, $1\leq |\al|\leq I-1$, $l\geq 1$ and $|\gamma|=|\al|$}. One has $|\beta|=0$. Then it holds
	
	\vspace{2mm}
	
	For $l=1$, one has $m+|\al|=I-1$. Consequently,
	\[
	\|\pt u\,\pt^m\pa_x^\al v\|\leq C\|\pt u\|_\infty\|\pt^m\pa_x^\al v\|\leq C\norm{\pt u}_{s-1}\norm{v}_{I-1}.
	\]
	
	For $2\leq l\leq s-1$, one has $m+|\al|+1=I-l+1\leq I-1$, then
	\[
	\|\pt^l u\,\pt^m\pa_x^\al v\|\leq C\|\pt^l u\|_1\|\pt^m\pa_x^\al v\|_1\leq  C\norm{\pt u}_{s-1}\norm{v}_{I-1}.
	\]
	
	For $l=s$, one has necessarily $m=|\al|=0$ and thus $I=s$. Then
	\[
	\|\pt^s u\,v\|\leq C\|\pt^s u\|\|v\|_\infty\leq C\norm{\pt u}_{s-1}\|v\|_{s-1}\leq C\norm{\pt u}_{s-1}\norm{v}_{I-1}.
	\]
	
	\vspace{3mm}
	
	\noindent \underline{Case D: $k\geq 1$, $1\leq |\al|\leq I-1$, $l\geq 1$ and $|\gamma|\leq |\al|-1$}. One has $|\beta|\geq 1$. Then apparently $l+|\beta|\geq 2$. Consequently,
	
	\vspace{2mm}
	
	For the case $2\leq l+|\beta|\leq s-1$, one has $m+|\gamma|\leq I-2$. Then
	\[
	\|\pt^l\pa_x^\beta u\,\pt^m\pa_x^\gamma v\|\leq C\|\pt^l\pa_x^\beta u\|_1\|\pt^m\pa_x^\gamma v\|_1 \leq C\norm{\pt u}_{s-1}\norm{v}_{I-1}.
	\]
	
	For the case $l+|\beta|=s$, one has necessarily $I=s$ and $m+|\gamma|=0$. Then,
	\[
	\|\pt^l\pa_x^\beta u\,v\|\leq C\|\pt^l\pa_x^\beta u\|\|v\|_\infty\leq C\norm{\pt u}_{s-1}\|v\|_{s-1}\leq C\norm{\pt u}_{s-1}\norm{v}_{I-1}.
	\]
	
	\noindent Combining all the cases, one has \eqref{Moser00}.
	
	We then prove \eqref{Moser11}. Notice the fact that
	\[
	\|u\pt^k\pa_x^\al v\|\leq C\|u\|_\infty\|\pt^k\pa_x^\al v\|\leq C\norm{u}_{s}\|\pt^k v\|_{|\al|},
	\]
	combining \eqref{Moser00}, one has \eqref{Moser11}. The proof for \eqref{Mosertt} is similar to the Case C treated above, we omitted it here.
\end{proof}

The next inequality concerns the estimates for composite functions.
\begin{lemma} (See \cite{Peng2015})
	Let $f$ be a smooth function and $v \in B_{s, T}(\mathbb{K}^3)$. Then
	\begin{equation*}
		\|\partial_{t}^{k} \partial_{x}^{\alpha} f(v)\| \le C\norm{\partial_{t} v}_{s-1}, \quad \forall\; k \geqslant 1,\quad k+|\alpha| \le s.
	\end{equation*}
	where the constant $C$ may depend continuously on $\|v\|_{s}$ and the given function $f$.
\end{lemma}

\subsection{Results on Euler-Maxwell system}
\begin{theorem}\label{theorem1}(Global existence and uniform estimates)
	Let $s\ge 3$ and $q\ge s+3$ be integers. There exist constants $\delta>0$ and $C>0$ independent of $\eps$ such that if
	$$ \dive E_0^\eps=b(x)-n_0^\eps,\quad \dive B_0^\eps=0, \quad \|(n_0^\eps-n_e, \eps u_0^\eps,E_0^\eps-E_e,B_0^\eps-B_e)\|_s\le\delta,$$
	system \eqref{EM2} admits a unique global-in-time solution $(n^{\eps}, u^{\eps},E^{\eps},B^{\eps})$ satisfying
	\begin{eqnarray}\label{2-1}
		&&\norm{(n^{\eps}(t)-n_e,\eps u^{\eps}(t),E^{\eps}(t)-E_e,B^{\eps}(t)-B_e)}_s^2\nonumber\\
		&&+\int_0^t \left(\norm{\left(n^{\eps}(\tau)-n_e,u^{\eps}(\tau)\right)}_s^2+\norm{E^{\eps}(\tau)-E_e}_{s-1}^2+\norm{\D\times{B^{\eps}(\tau)}}_{s-2}^2\right) \mathrm{d}{\tau}\nonumber\\
		&\le& C\|(n_0^\eps-n_e,\eps u_0^\eps,E_0^\eps-E_e,B_0^\eps-B_e)\|_s^2, \quad \forall\, t\geq 0.
	\end{eqnarray}
\end{theorem}
\begin{theorem}\label{theorem2}(Global convergence)
	Let $(n^{\eps},u^{\eps},E^{\eps},B^{\eps})$ be the global solution obtained in Theorem \ref{theorem1}. Assume $(\bar{n}_0-n_e,\bar{E}_0-E_e)\in H^s\times H^s$ and as $\eps\rightarrow 0$,
	\begin{center}
		$\left(n_0^{\eps}-n_e,E_0^{\eps}-E_e,B_0^{\eps}-B_e\right) \rightharpoonup\left(\bar{n}_0-n_e,\bar{E}_0-E_e,0\right),\quad$weakly in $H^s$.
	\end{center}
	Then there exist functions $(\bar{n},\bar{u},\bar{E})$ with $(\bar{n}-n_e,\bar{E}-E_e)\in L^{\infty}\left(\mathbb{R}^+;H^s\right)$ and $\bar{u}\in L^{2}\left(\mathbb{R}^+;H^s\right)$ such that as $\eps\to 0$,
\begin{align}\label{cov}
	\begin{split}
		(n^{\eps}-n_e,E^{\eps}-E_e)\stackrel{\ast}{\rightharpoonup} (\bar{n}-n_e,\bar{E}-E
		_e),\quad & \text{weakly-*}\,\,\text{in}\,\, L^{\infty}\left(\mathbb{R}^+;H^s\right),\\
		B^{\eps}-B_e \stackrel{\ast}{\rightharpoonup} 0, \quad &\text{weakly-*}\,\,\text{in}\,\, L^{\infty}\left(\mathbb{R}^+;H^s\right),\\
		u^\eps \rightharpoonup \bar{u},\quad &\text{weakly}\,\,\text{in}\,\, L^{2}\left(\mathbb{R}^+;H^s\right),
	\end{split}
\end{align}
	where $\bar{E}=\D\bar{\phi}$ and $(\bar{n},\bar{\phi})$ is the unique global smooth solution to the drift-diffusion system \eqref{EM-drift-diffusion} with the initial condition $\bar{n}(0,x)=\bar{n}_{0}$, and $\bar{u}$ satisfies \eqref{baru}.
\end{theorem}
\begin{theorem}\label{convergence_rate}(Global-in-time convergence rate)
	Let $\mathbb{K}=\T$ and the conditions in Theorems \ref{theorem1} and \ref{theorem2}  hold. Let $(n^\eps,u^\eps,E^\eps,B^\eps)$, $(\bar{n},\bar{u},\bar{E})$ and $(n_e,E_e,B_e)$ be the unique smooth solutions to \eqref{EM2}-\eqref{EM2initial}, \eqref{EM-limit} and \eqref{PHT}, respectively.  Then for any positive constant $p>0$ independent of $\eps$, if
	\begin{equation*}
		\eps \|u_{0}^{\eps}\|_{s-1}+\|E_{0}^{\eps}-\bar{E}_{0}\|_{s-1}+\|B_{0}^{\eps}-B_{e}\|_{s-1} \leq C\eps^{p},
	\end{equation*}
	then for all $\eps\in (0,1]$, one has for $p_1:=\min\{1,p\}$,
	\begin{align}
		& \sup_{t\in \mathbb{R}^{+}}\left(\|(n^\eps-\bar{n},\eps(u^\eps-\bar{u}), E^\eps-\bar{E},B^\eps-\bar{B})(t)\|_{s-1}^{2}\right) \nonumber\\
		&+ \int_{0}^{+\infty}\left(\|(n^\eps-\bar{n},u^\eps-\bar{u}, E^\eps-\bar{E})(t))\|_{s-1}^{2}+\|\D\times B^\eps(t)\|_{s-2}^{2}\right) \mathrm{d}t\leq C\eps^{2 p_{1}}.\nonumber
	\end{align}
\end{theorem}

\subsection{Results on Euler-Poisson system}
\begin{theorem}\label{theorem5.1}(Global-in-time existence and convergence)
	Let $s\ge 3$ and $q\ge s+3$ be integers.  There exist constants $\delta>0$ and $C>0$ independent of $\eps$ such that if
	$$ \|(n_0^\eps-n_e, \eps u_0^\eps)\|_s\le\delta,$$
	system \eqref{EPmain} admits a unique global-in-time solution $(n^{\eps}, u^{\eps},\phi^{\eps})$ satisfying:
	\begin{eqnarray}\label{5.1.1}
		&&\norm{\left(n^\eps(t,\cdot)-n_e,\eps u^\eps(t,\cdot),\nabla\phi^\eps(t,\cdot)-\nabla\phi_e\right)}_s^2
		+\int_{0}^{t}\norm{(n^\eps(\tau,\cdot)-n_e,u^\eps(\tau,\cdot))}_s^2 \mathrm{d}{\tau}\nonumber\\
		&\le& C\|(n_0^\eps-n_e,\eps u_0^\eps)\|_s^2, \quad \forall\, t\geq 0.
	\end{eqnarray}
	Furthermore, assume $\bar{n}_0-n_e\in H^s$ and as $\eps\rightarrow 0$,
	\begin{equation*}
		n_0^\eps-n_e \rightharpoonup \bar{n}_0-n_e,
		\qquad  \text{weakly}\; \text{in}\;  H^s,
	\end{equation*}
	then there exist functions $(\bar{n},\bar{u},\bar{\phi})$ with $(\bar{n}-n_e,\D\bar{\phi}-E_e)\in L^{\infty}\left(\mathbb{R}^+;H^s\right)$ and $\bar{u}\in L^{2}\left(\mathbb{R}^+;H^s\right)$ such that as $\eps\to 0$,
	\begin{align}\label{5.1.2}
		\begin{split}
		n^\eps-n_e\overset{*}{\rightharpoonup}\bar{n}-n_e,
		\quad&\text{weakly-}* \; \text{in} \; L^\infty(\R^+;H^s),\\
		\nabla\phi^\eps-E_e\overset{*}{\rightharpoonup}\nabla\bar{\phi}-E_e,
		\quad&\text{weakly-}* \; \text{in} \; L^\infty(\R^+;H^s),\\
u^\eps \rightharpoonup \bar{u}, \quad &\text{weakly} \;
		\text{in} \; L^2(\R^+;H^s),
		\end{split}
	\end{align}
	where $(\bar{n},\bar{\phi})$ is the unique global smooth solution to \eqref{EM-drift-diffusion} and $\bar{u}$ satisfies \eqref{baru}.
\end{theorem}

\begin{theorem}\label{theorem5.2}(Global-in-time convergence rate)
	Let the conditions in Theorem \ref{theorem5.1} hold. Let $(n^\eps,u^\eps,\phi^\eps)$, $(\bar{n},\bar{u},\bar{\phi})$, and $(n_e,E_e)$ be the unique smooth solutions to \eqref{EPmain}, \eqref{PHT} and \eqref{EM-drift-diffusion}-\eqref{baru}, respectively. Assume for any given positive constant $q>0$ independent of $\eps$, it holds
	\begin{equation*}
		\|n_0^\eps-n_e\|_{s-1}+\eps \|u_0^\eps\|_{s-1}\le C\eps^q,
	\end{equation*}
	then for all $\eps\in (0,1]$, one has for $q_1:=\min\{q,1\}$:
	\begin{eqnarray}\label{EPerror}
		&\quad&\sup_{t\in\R^+}\left(\|(n^\eps(t)-\bar{n},\eps (u^\eps-\bar{u}),\nabla \phi^\eps-\nabla\bar{\phi})(t)\|_{s-1}^2\right)\nonumber\\
		&&+\int_0^{+\infty}\left(\|(n^\eps(t)-\bar{n},\eps (u^\eps-\bar{u}),\nabla \phi^\eps-\nabla\bar{\phi})(t)\|_{s-1}^2\right)\mathrm{d}t\le C\eps^{2q_1}.
	\end{eqnarray}
\end{theorem}

\section{Global convergence analysis for Euler-Maxwell system}

This section is devoted to the uniform global-in-time estimates regarding $\eps$ of solutions near non-constant equilibrium states to Euler-Maxwell system, based on which the global-in-time convergence analysis when $\eps\to 0$ is then carried out. For convenience, in this section, we drop the superscript of $\eps$. Let $T>0$, $(n,u,E,B)$ be the smooth solution to \eqref{EM2} defined on $[0,T]$ and $(n_e,E_e, B_e)$ be the steady solution to  \eqref{PHT}. We denote afterwards
\begin{eqnarray*}
    N&=&n-n_e, \quad F=E-E_e, \quad G=B-B_e,\nonumber\\
    U&=&\left(\begin{array}{c}N \\u \end{array}\right),
\quad W=\left(\begin{array}{c}N \\ \eps{u} \\ F \\ G\end{array}\right), \quad W_{0}^\eps=\left(\begin{array}{c}n_{0}^\eps-n_e \\ \eps u_0^\eps \\ E_{0}^\eps-E_e \\ B_{0}^\eps-B_e\end{array}\right).
\end{eqnarray*}
In addition, we introduce the functionals
$$W_T=\sup_{t\in [0,T]}\norm{W(t,\cdot)}_s^2.$$
which we assume to be sufficiently small. Thus, Proposition \ref{extcequil} naturally leads to the fact that there exist positive constants $n_1$, $n_2$ and $h_1$, such that
\begin{equation}\label{3.1}
	n_1<n<n_2, \quad h^\prime(n)\ge h_1, \,\,\text{for } \,\,\forall\,n>0.
\end{equation}

For readers' convenience, we here state our strategies in establishing uniform estimates of solutions regarding $\eps$. The proof can be divided into two steps:
\begin{itemize}
	\item Taking advantage of the anti-symmetric structure of the Euler-Maxwell system, one obtains that the solutions with only time derivatives are bounded by cubic terms of energy functionals (See Lemma \ref{lemma3.1}).
	\item The regular symmetrizable hyperbolicity for Euler-Maxwell system leads to the fact that solutions with mixed space and time derivatives are bounded by cubic terms of energy functionals or quadratic terms of those with higher order time derivatives but lower order space ones (See Lemmas \ref{lemma3.2}-\ref{Lemma3.3}). These enable us to perform an induction argument to convert space derivatives of solutions order-by-order to time derivatives.
\end{itemize}

The next lemma concerns estimates for solutions with only time derivatives.

\begin{lemma}\label{lemma3.1} Let $0\leq l\leq s$ be integers. Then it holds
	\begin{equation}\label{estimptu}
	\|\pt^l W(T)\|^2+2n_1\int_0^T \|\pt^{l} u(t)\|^{2}\mathrm{d}t\le \|\pt^l W(0)\|^2+\cubict.
	\end{equation}
 Besides, for integers $0\leq k\leq s-1$, one obtains
	\begin{eqnarray}\label{ptestim}
		\int_0^T \|(\pt^{k} N,\pt^{k} F)(t)\|^{2}\mathrm{d}t\leq C\| W(0)\|_s^2+\cubict.
	\end{eqnarray}
\end{lemma}
\begin{proof} The proof is based on the anti-symmetric structure of Euler-Maxwell system, without which many quadratic terms of energy functionals will inevitably appear. Subtracting \eqref{PHT} correspondingly from \eqref{EM2} leads to
\begin{equation}\label{EM30}
	\begin{cases}
		\pt N+u \cdot \D N+n \dive u+u\cdot \D n_e=0,\\
		\eps^2\pt u+\eps^2 ((u \cdot \D) u)+\D(h(n)-h(n_e))=-F-u-\eps u \times B, \\
		\eps\pt F-\D \times G =\eps nu, \\
		\eps \pt G+\D \times F=0, \\
		\dive F=-N , \quad \dive {G}=0.
	\end{cases}
\end{equation}
The difference of the enthalpy functions in \eqref{EM30}$_2$ can be rewritten into
	\begin{equation*}
		\D(h(n)-h(n_e))=h^{\prime}(n) \D N+\D h^{\prime}(n_e) N+r\left(n_e, N\right),
	\end{equation*}
	with the remaining term defined as
	\[
	r(n_e, N)=\left(h^{\prime}(n)-h^{\prime}(n_e)-h^{\prime \prime}(n_e) N\right) \D n_e.
	\]
By Taylor's expansion for $h'(n)$ at $n=n_e$, one obtains that $r(n_e, N)$ is actually an $O(N^2)$ term. For $U=(N,u^\top)^\top$, equations \eqref{EM30}$_1$--\eqref{EM30}$_2$ can be written into
\begin{equation}\label{sym1}
		D_0(\eps)\pt U+\sum_{j=1}^3A_j(n,u)\pa_{x_j} U+\hat{L}(n_e)U=\hat{f},
	\end{equation}
in which for $j=1,2,3$,
\begin{equation}\label{def}
	\begin{split}
D_0(\eps)=&\text{diag}(1,\eps^2\mathbb{I}_{3}), \quad\,\,\,\, A_{j}(n, u)=\left(\begin{array}{cc}u_{j} & n e_{j}^\top \\
	h'(n) e_{j} & \eps^2 u_{j} \mathbb{I}_{3}
\end{array}\right),\\
\hat{L}(n_e)=&\left(
	\begin{matrix}
		0& (\D n_e)^\top\\
		\D h^{\prime}(n_e)&0\\
	\end{matrix}
	\right),\quad
	\hat{f}=-\left(
	\begin{matrix}
		0\\
		F+u+\eps u\times B+r(n_e,N)
	\end{matrix}
	\right).
	\end{split}
\end{equation}
Here $\{e_j\}_{j=1}^3$ denotes the canonical basis of $\R^3$ and $\mathbb{I}_{3}$ denotes the $3\times 3$ unit matrix.

Now we introduce the symmetrizer $A_0(n)$ as well as $\widetilde{A}_{j}$ as follows:
\begin{equation}\label{symmetrizer}
A_{0}(n):=\left(\begin{array}{cc}h^{\prime}(n) & 0 \\ 0 & n \mathbb{I}_{3}\end{array}\right),\,\, \widetilde{A}_{j}(n, u):=A_{0}(n) A_{j}(n, u)=
\left(\begin{matrix}
	h^\prime(n)u_j& P^\prime(n)e_j^\top\\
	P^\prime(n)e_j & \eps^2 n u_j \mathbb{I}_{3}
\end{matrix}
\right).
\end{equation}
It is clear from \eqref{3.1} that $A_{0}(n)$ is symmetric and positive definite while $\widetilde{A}_{j}(n, u)$ is symmetric. This implies the symmetrizable hyperbolicity of the system \eqref{EM30}.

For integers $0\leq l\leq s$, applying $\pt^l $ to \eqref{sym1}, one obtains,
	\begin{equation}\label{sympt}
		D_0(\eps)\pt^{l+1}U+\sum_{j=1}^{3} A_{j}(n, u) \pt^l\pa_{x_{j}} U+ \hat{L}(n_e) \pt^l U=\pt^l \hat{f}+g_t^{l},
	\end{equation}
	with the commutators defined as
	\[
	g_t^l:=-\pt^l\left(\sum_{j=1}^{3} A_{j}(n, u) \pa_{x_{j}} U\right)+\sum_{j=1}^{3} A_{j}(n, u) \pt^l\pa_{x_{j}} U.
	\]
	Taking the inner product of \eqref{sympt} with $2A_0(n)\pt^l U$ in $L^2$ yields
	\begin{eqnarray}
		\frac{\mathrm{d}}{\mathrm{d} t}\left< D_0(\eps)A_{0}(n) \pt^l U, \pt^l U\right>&=&\left<D_0(\eps)\pt A_0(n)\pt^l U, \pt^l U \right>+\left<B(U,\D U)\pt^l U, \pt^l U\right>\nonumber\\
		&&+2\left< A_{0}(n) g_t^l, \pt^l U\right>+2\left< A_{0}(n) \pt^l \hat{f}, \pt^l U\right>:=\sum_{j=1}^4 I_t^j,\nonumber
	\end{eqnarray}
	with the natural correspondence of $\{I_t^j\}_{j=1}^4$. Here, $B(U,\D U)$ is defined as
	\begin{equation}\label{defBno}
		\begin{split}
			B(U,\D U):=\sum_{j=1}^{3} \pa_{x_{j}} \widetilde{A}_{j}(n, u)-2 A_{0}(n) \hat{L}(n_e):=\left(\begin{array}{cc} B_{11} & B_{12}  \\ B_{21} & B_{22} \end{array}\right),
		\end{split}
		\end{equation}
		where $B_{11}=\dive\left(h^{\prime}(n) u\right)$, $B_{22}=\dive(\eps^2  nu) \mathbb{I}_{3}$ and
		\[
		B_{12}=\left(\D P^{\prime}(n)-2 h^{\prime}(n) \D n_e\right)^\top, \quad B_{21}=\D P^{\prime}(n)-2 n \D h^{\prime}(n_e).
		\]
	The matrix $B(U,\D U)$ is anti-symmetric at the equilibrium state $(U_e,\D U_e)$. Indeed, at the equilibrium state $n=n_e$,
	\begin{eqnarray}\label{antisymm}
		B_{21}(n_e,\D n_e)&=&\D P^{\prime}(n_e)-2 n_e \D h^{\prime}(n_e)\nonumber\\
		&=&(P^{\prime\prime}(n_e)-2n_eh^{\prime\prime}(n_e))\D n_e\nonumber\\
		&=&(P^{\prime\prime}(n_e)-2(P^{\prime\prime}(n_e)-h^\prime(n_e))\D n_e\nonumber\\
		&=&-(P^{\prime\prime}(n_e)-2h^\prime(n_e))\D n_e=-B_{12}(n_e,\D n_e)^\top.
	\end{eqnarray}
Consequently, one obtains $B_{12}(U,\D U)^\top+B_{21}(U,\D U)$ is an $O(N)$ term by using Taylor's expansions at the equilibrium. With these on hand, we are ready to estimate $\{I_t^j\}_{j=1}^4$ term by term as follows.

    \vspace{2mm}

	\par \underline{Estimate of $I_t^1$}: It follows from the definition of $I_t^1$ that
	\begin{eqnarray}\label{It1}
		|I_t^1|\leq  C\|\pt A_{0}(n)\|_{\infty}\|\pt^l U\|^2\leq C\|\pt n\|_{s-1} \|\pt^l U\|^2\leq  \cubic.
	\end{eqnarray}
	
	\par \underline{Estimate of $I_t^2$}:
	It is clear that by the notation introduced in \eqref{defBno},
	\begin{eqnarray}
		I_t^2&=&\left<B_{11}\pt^l N,\pt^l N\right>+\left<B_{22}\pt^l u,\pt^l u\right>+\left<(B_{12}(n,\D n)^\top+B_{21}(n,\D n))\pt^l N,\pt^l u\right>,\nonumber
	\end{eqnarray}
	where the first two terms are bounded by $\cubic$. Due to \eqref{antisymm}, it holds
	\begin{equation*}
		\left|\left<(B_{12}(n,\D n)+B_{21}(n,\D n)^\top)\pt^l N, \pt^l u\right>\right|\leq C\norm{N}_s\|\pt^l u\|\|\pt^l N\|\leq \cubic,
	\end{equation*}
	and consequently, one has
	\begin{equation}\label{It2}
		|I_t^2|\leq \cubic.
	\end{equation}

	\par \underline{Estimate of $I_t^3$}: Notice $g_t^0=0$. For $l\geq 1$, by \eqref{Mosertt}, one has
	\begin{eqnarray}\label{It3}
		\left|I_t^3\right|&\leq&2\left|\left< u \cdot \D \pt^l N- \pt^l (u \cdot \D N), h^{\prime}(n) \pt^l N\right>\right| \nonumber\\
		&&+2{\eps}^2\left|\left< (u \cdot \D )\pt^l u- \pt^l((u \cdot \D) u), n \pt^l u\right>\right|\nonumber \\
		&&+2\left|\left< n\dive \pt^l u- \pt^l (n \dive u), h^{\prime}(n) \pt^l N\right>\right| \nonumber\\
		&&+2\left|\left< h^{\prime}(n) \D \pt^l N-
		\pt^l \left(h^{\prime}(n) \D N\right), n \pt^l u\right>\right|\leq \cubic.
	\end{eqnarray}
	
	\par \underline{Estimate of $I_t^4$}: By the definition of $\hat{f}$ in \eqref{sym1}, one has
	\begin{eqnarray}
		I_t^4=-\left<2n\pt^l F,\pt^lu\right>-\left<2n\pt^l u, \pt^l u\right>-\left<2\eps n\pt^l(u\times B), \pt^l u\right>-\left<2n \pt^l r(n_e,N), \pt^l u\right>.\nonumber
	\end{eqnarray}
	It is clear that by noting \eqref{3.1},
	\[
	\left<2n\pt^l u, \pt^l u\right>\ge 2n_1\|\pt^l u\|^2.
	\]
	Also, since $\left<n \pt^l (u\times B_e), \pt^l u\right>=0$, one obtains
	\begin{equation*}
		\left|\left<2\eps n\pt^l (u\times B), \pt^l u\right>\right|= 2\eps\left|\left<n(\pt^l(u\times B)-\pt^l(u\times B_e)), \pt^l u\right>\right|
		\leq C\eps\norm{u}_s^2\norm{G}_s.
	\end{equation*}
	A direct calculation shows that
	\[
	|\left<2n \pt^l r(n_e,N), \pt^l u\right>|\leq C\norm{N}_s^2\|\pt^l u\|\leq \cubic.
	\]
	Combining all these estimates, we arrive at
	\begin{equation}\label{It4naka}
		I_t^4\leq -\left<2n\pt^l F, \pt^l u\right>-2n_1\|\pt^l u\|^2+\cubic.
	\end{equation}
	
	In order to control $\left<2n\pt^l F, \pt^l u\right>$, we apply $\pt^l $ to \eqref{EM30}$_3$--\eqref{EM30}$_5$, leading to
	\begin{equation*}
		\left\{
		\begin{array}{l}
			\pt^{l+1} F-\eps^{-1}\D \times \pt G=\pt^l (nu),  \\
			\pt^{l+1} G+\eps^{-1}\D \times \pt F=0, \\
	        \dive  \pt^l F=-\pt^l N, \quad \dive \pt^l G=0.
	\end{array}
		\right.
	\end{equation*}
Taking the inner product of the first two equations with $(\pt^l F, \pt^l G)$ yields
	\begin{equation}\label{maxwell-t}
		\frac{\mathrm{d}}{\mathrm{d}t}\left(\|\pt^l F\|^2+\|\pt^l G\|^2\right)-2\left< \pt^l (nu),\pt^l F\right>=0.
	\end{equation}
	Combining the above with \eqref{It4naka}, one obtains
	\begin{eqnarray}\label{It4}
		&&I_t^4+\frac{\mathrm{d}}{\mathrm{d}t}\left(\|\pt^l F\|^2+\|\pt^l G\|^2\right)+2n_1\|\pt^l u\|^2 \nonumber\\
		&\leq&\cubic +C\left|\left< \pt^l F,\pt^l(nu)-n\pt^l u\right>\right|\leq \cubic.
	\end{eqnarray}
	Finally, combining \eqref{It1},   \eqref{It2}, \eqref{It3} and \eqref{It4} yields
	\begin{equation}\label{base1}
	\frac{\mathrm{d}}{\mathrm{d} t}\left(\left< D_0(\eps)A_{0}(n) \pt^l U, \pt^l U\right>+\|(\pt^l F,\pt^l G)\|^2\right)+2n_1\|\pt^l u\|^2\leq \cubic.
	\end{equation}
    Integrating the above over $[0,T]$ and noticing that  $\left< D_0(\eps)A_{0}(n) \pt^l U, \pt^l U\right>$ is equivalent to $\|\pt^l W\|^2$, one obtains \eqref{estimptu}.

	Next, applying $\pt ^k$ to \eqref{EM30}$_2$ with $0\leq k\leq s-1$, one has
	\begin{equation*}
		\pt^k F+\pt^k\D(h(n)-h(n_e))=-\eps^2\pt^{k+1} u-\pt^k u-\pt^k (\eps u\times B+\eps^2(u\cdot\D) u).
	\end{equation*}
	Taking the inner product of the above with $\pt^k F$ yields
	\begin{eqnarray}\label{bEqula0}
		&&	\|\pt^k F\|^2+\left<\pt^k\D(h(n)-h(n_e)), \pt^k F\right>\nonumber\\
		&=&-\left<\eps^2\pt^{k+1} u+\pt^k u+\pt^k (\eps u\times B+\eps^2(u\cdot\D) u),  \pt^k F\right>\nonumber\\
		&\leq& \dfrac{1}{2}\|\pt^k F\|^2+C\eps^4\|\pt^{k+1} u\|^2+C\|\pt^k u\|^2+\cubic.
	\end{eqnarray}
	It remains to estimate the last term on the left hand side of \eqref{bEqula0}. Notice that
	\[
	h(n)-h(n_e)=N\int_0^1 h'(n_e+\theta N)\mathrm{d}\theta:=N\int_0^1 h'(\tilde{n}^\theta)\mathrm{d}\theta,
	\]
	with the natural correspondence of $\tilde{n}^\theta$. Consequently,
	\begin{eqnarray}\label{3.17}
		&&\left<\pt^k\D(h(n)-h(n_e)), \pt^k F\right>\nonumber\\
		&=&-\left<\pt^k(h(n)-h(n_e)), \pt^k \dive F\right>\nonumber\\
		&=&\left<\int_0^1 h'(\tilde{n}^\theta)\mathrm{d}\theta\pt^k N, \pt^k N\right>+\left<\pt^k\left(N\int_0^1 h'(\tilde{n}^\theta)\mathrm{d}\theta\right)-\int_0^1 h'(\tilde{n}^\theta)\mathrm{d}\theta\pt^k N, \pt^k N\right>\nonumber\\
		&\geq& h_1\|\pt^k N\|^2-\cubic,
	\end{eqnarray}
	where we have used \eqref{Mosertt} and the fact $\pt\tilde{n}^\theta=\theta\pt N$. Combining \eqref{bEqula0}, it holds
	\begin{eqnarray}\label{base2}
		\dfrac{1}{2}\|\pt^k F\|^2+h_1\|\pt^k N\|^2\leq C\eps^4\|\pt^{k+1} u\|^2+C\|\pt^k u\|^2+\cubic.
	\end{eqnarray}
	Integrating the above over $[0,T]$ and combining \eqref{estimptu} yield \eqref{ptestim}.
\end{proof}

For simplicity, for any multi-index $\al\in\mathbb{N}^3$, we denote
\[
U_\al=\pa_x^\al U, \quad W_\al=\pa_x^\al W, \quad  (N_\al, u_\al, F_\al, G_\al)=(\pa_x^\al N,\pa_x^\al u, \pa_x^\al F,\pa_x^\al G ).
\]
In the following, we denote by $\mu>0$ a sufficiently small constant, of which the value is determined in \eqref{mu}.

The next lemma concerns estimates for solutions with space derivatives.

\begin{lemma}\label{lemma3.2} Let $0\leq k\leq s-1$ be integers and multi-indices $\al\in\mathbb{N}^3$ satisfying $1\leq |\al|\leq s$ and $|\al|+k\leq s$, then it holds
	\begin{eqnarray}\label{ptinduction}
		&&\frac{\mathrm{d}}{\mathrm{d}t}\left(\left< D_0(\eps)A_0(n)\pt^k U_\al,\pt^k U_\al\right>+\|\pt^k F_\al\|^2+\|\pt^k G_\al\|^2\right)+2n_1\|\pt^k u_\al\|^2\nonumber\\
		&\leq&  C\mu\norm{u}_{s}^2+C\|\pt^k N\|_{|\al|}^2+C\|\pt^k F\|_{|\al|-1}^2+C\norm{U}_s^2\norm{W}_s.
	\end{eqnarray}
\end{lemma}
\begin{proof} We start with \eqref{EM30}. Notice that
	\begin{equation}\label{h}
		\D(h(n)-h(n_e))=h^\prime(n)\D n-h^\prime(n_e)\D n_e= h^\prime(n)\D N+ N \D n_e\int_0^1 h^{\prime\prime}(\tilde{n}^\theta)\mathrm{d}\theta.
	\end{equation}
Then \eqref{EM30}$_1$--\eqref{EM30}$_2$ can be written into
	\begin{equation}\label{Sym}
		D_0(\eps)\pt U+\sum_{j=1}^{3} A_{j}(n, u) \pa_{x_{j}} U+L(\D n_e; n,n_e) U=f,
	\end{equation}
where $D_0(\eps)$ and $A_j(n,u)$ is defined in \eqref{def} and
\[
L(\D n_e;n,n_e)=\left(\begin{array}{cc}
	0 &(\D n_e)^\top \\
	\displaystyle\int_0^1 h^{\prime\prime}(\tilde{n}^\theta)\mathrm{d}\theta\D n_e & 0
\end{array}\right),\quad
f=-\left(\begin{array}{c}
	0 \\ F+u+\eps u \times B
\end{array}\right).
\]

	For integers $0\leq k\leq s-1$ and multi-indices $\al\in\mathbb{N}^3$ with $k+|\al|\leq s$ and $|\al|\geq 1$, applying mixed space and time derivatives $\pt^k\pa_x^\al$ to \eqref{Sym}, one obtains
	\begin{equation}\label{Symla}
		D_0(\eps)\pt^{k+1}U_\al+\sum_{j=1}^{3} A_{j}(n, u) \pt^k\pa_{x_{j}} U_\al+ \pa_x^\al (L(\D n_e) \pt^k U)=\pt^k\pa_x^\al f+g_t^{k,\al},
	\end{equation}
	with the commutators defined as
	$$
	g_t^{k,\al}=\sum_{j=1}^{3}\left(A_{j}(n, u) \pt^k \pa_{x_{j}} U_{\al}- \pt^k \pa_{x}^\al\left(A_{j}(n, u) \pa_{x_{j}} U\right)\right).
	$$
	Taking the inner product of \eqref{Symla} with $2A_0(n)\pt^k U_\al$ in $L^2$ yields
	\begin{align}\label{EnergyEqula}
		&\frac{\mathrm{d}}{\mathrm{d} t}\left< D_0(\eps)A_{0}(n) \pt^k U_\al, \pt^k U_\al\right>\nonumber\\
		=&\left<\dive A(n,u) \pt^k U_\al, \pt^k U_\al \right>-2\left<A_0(n)\pa_x^\al (L(\D n_e) \pt^k U), \pt^k U_\al\right>\nonumber\\
		&+2\left< A_{0}(n) g_t^{k,\al}, \pt^k U_{\al}\right>+2\left< A_{0}(n) \pt^k\pa_{x}^{\al} f, \pt^k U_{\al}\right>:=\sum_{j=1}^4I_{k,\al}^j,\nonumber
	\end{align}
with the natural correspondence of $\{I_{k,\al}^j\}_{j=1}^4$, and $\dive A(n,u)$ is defined as
	\begin{equation*}
		\dive A(n,u)=D_0(\eps)\pt A_0(n)+\sum_{j=1}^3 \pa_{x_j} \tilde{A_j}(n,u).
	\end{equation*}
	
	Similar to the treatment in \eqref{It1}, one has
	\[
	\left|\left< D_0(\eps)\pt A_{0}(n) \pt^k U_\al, \pt^k U_\al\right>\right|\leq \cubic.
	\]
	Besides, for the term containing $\pa_{x_j} \tilde{A_j}(n,u)$, it holds that for a certain $j$,
	\begin{eqnarray}
		\left<\pa_{x_j}\tilde{A_j}(n,u)\pt^k U_\al,\pt^k U_\al\right>&=&\left<\pa_{x_j}(h^\prime(n)u_j) \pt^k N_\al, \pt^k N_\al \right>
		+\eps^2\left<\pa_{x_j}(nu_j) \pt^k u_\al, \pt^k u_\al \right>\nonumber\\
		&&+2\left<\pa_{x_j}(P^\prime(n)) \pt^k N_\al, \pt^k \pa_x^\al u_j\right>.\nonumber
	\end{eqnarray}
	It is clear that the first two terms can be controlled by $\cubic$.
	For the remaining term,  since $\|\pa_{x_j}(P^\prime(n))\|_\infty$ is bounded  but not small, we can only obtain the following quadratic estimates
	\begin{equation*}
		2 \left|\left<\pa_{x_j}(P^\prime(n)) \pt^k N_\al, \pt^k \pa_x^\al(u_j) \right>\right|\leq C\|\pt^k u_\al\|\|\pt^k N_\al\|\leq \mu \norm{u}_s^2+ C\|\pt^k N\|_{|\al|}^2.
	\end{equation*}
	These estimates lead to
	\begin{equation}\label{Ix1}
		|I_{k,\al}^1|\leq \mu \norm{u}_s^2+ C\|\pt^k N\|_{|\al|}^2+\cubic.
	\end{equation}
Since $\D n_e$ is bounded in $H^s$ but not small compared to the case of constant equilibrium states, a direct calculation shows
	\begin{equation}\label{Ix2}
		|I_{k,\al}^2|\leq \|\D n_e\|_{s-1}\|\pt^k N\|_{|\al|}\|\pt^k u_\al\|\leq \mu \norm{u}_s^2+ C\|\pt^k  N\|_{|\al|}^2.
	\end{equation}
	
	As to $I_{k,\al}^3$, one obtains
	\begin{eqnarray}\label{Ix3}
		|I_{k,\al}^3|	&\leq& C\left|\left< u \cdot \D \pt^k N_ \al- \pa_{x}^{\al}\pt^k(u \cdot \D N), h^{\prime}(n) \pt^kN_ \al\right>\right| \nonumber\\
		&&+C{\eps}^2\left|\left< (u \cdot \D )\pt^k u_\al- \pt^k\pa_{x}^{\al}((u \cdot \D) u), n \pt^ku_\al\right>\right|\nonumber \\
		&&+C\left|\left< n\dive \pt^k u_\al- \pt^k\pa_{x}^{\al}(n \dive u), h^{\prime}(n)\pt^k N_\al\right>\right| \nonumber\\
		&&+C\left|\left< h^{\prime}(n) \D \pt^k N_\al-
		\pt^k\pa_{x}^{\al}\left(h^{\prime}(n) \D  N\right), n \pt^ku_\al\right>\right|\nonumber\\
		&\leq&	\mu \norm{u}_s^2+C\|\pt^k N\|_{|\al|}^2+\cubic,
	\end{eqnarray}
	in which we have used the inequality \eqref{Moser00}.
	
	For $I_{k,\al}^4$, similar to \eqref{It4naka}, one obtains
	\begin{equation}\label{Ix4}
	I_{k,\al}^4\leq -\left<2n\pt^k F_\al, \pt^k u_\al\right>-2n_1\|\pt^k u_\al\|^2+\cubic.
	\end{equation}
	
	Combining estimates \eqref{Ix1}--\eqref{Ix4}, we arrive at
	\begin{eqnarray}\label{naka}
		&&\frac{\mathrm{d}}{\mathrm{d}t}\left(\left< D_0(\eps)A_0(n)\pt^k U_\al,\pt^k U_\al\right>\right)+2n_1\|\pt^k u_\al\|^2+2\left< \pt^k F_\al,n\pt^k u_\al\right>\nonumber\\
		&\leq&  \mu\norm{u}_{s}^2+\cubic.
	\end{eqnarray}
	
	Similarly as \eqref{maxwell-t}, \eqref{EM30}$_3$--\eqref{EM30}$_4$ imply
	\begin{equation*}
		\frac{\mathrm{d}}{\mathrm{d}t}\left(\|\pt^k F_\al\|^2+\|\pt^k G_\al\|^2\right)-2\left< \pt^k \pa_x^\al(nu),\pt^k F_\al\right>=0,
	\end{equation*}
	which further combining \eqref{naka} yields
	\begin{eqnarray}\label{naka3}
		&&\frac{\mathrm{d}}{\mathrm{d}t}\left(\left< D_0(\eps)A_0(n)\pt^k U_\al,\pt^k U_\al\right>+\|\pt^k F_\al\|^2+\|\pt^k G_\al\|^2\right)+2n_1\|\pt^k u_\al\|^2\nonumber\\
		&\leq&  \mu\norm{u}_{s}^2+C\|\pt^k N\|_{|\al|}^2+\cubic+2\left<\pt^k F_\al, \pt^k\pa_x^\al(nu)-n\pt^k u_\al\right>. \quad
	\end{eqnarray}

	Now we aim to control the last term of the above inequality. Noticing that
	\[
	\pt^k\pa_x^\al(nu)-n\pt^k u_\al=	\pt^k\pa_x^\al(Nu)-N\pt^k u_\al+\pt^k\pa_x^\al(n_eu)-n_e\pt^k u_\al,
	\]
	and consequently, by the Moser-type calculus inequalities,
	\[
	\left|\left<\pt^k F_\al, 	\pt^k\pa_x^\al(Nu)-N\pt^k u_\al\right>\right|\leq \cubic.
	\]
	Again since $\D n_e$ is not small, quadratic estimates are inevitable. Noticing $|\al|\ge 1$, without loss of generality, we may assume that $\al_1\neq 0$. We denote a multi-index $\al^\prime\in \mathbb{N}^3$ with $|\al^\prime|=|\al|-1$ and $\pa_{x_1}\pa_x^{\al^\prime}=\pa_x^\al$, then integration by parts gives
	\begin{eqnarray}\label{treatF}
		\left|\left<\pt^k F_\al, \pt^k\pa_x^\al(n_eu)-n_e\pt^k u_\al\right>\right|&=&\left|\left<\pt^k F_{\al^\prime},\pa_{x_1}(\pt^k\pa_x^\al(n_eu)-n_e\pt^ku_\al)\right>\right|\nonumber\\
		&\leq& \mu \norm{u}_s^2+ C\|\pt^k F\|_{|\al|-1}^2.
	\end{eqnarray}
	Combining the above two estimates and \eqref{naka3} yields \eqref{ptinduction}.
\end{proof}

\begin{lemma}\label{Lemma3.3}(Dissipative estimates for $N$ and $F$) Let $0\leq k\leq s-1$ be integers and $\al$ be multi-indices satisfying $|\al|\geq 1$ and $k+|\al|\leq s$, then it  holds
	\begin{eqnarray}\label{NFla}
		&&\|\pt ^k N\|_{|\al|}^{2} +\|\pt ^k F\|_{| a|-1}^{2}\nonumber\\
		& \le& C\|\pt ^k u\|_{|\al|-1}^{2}+C{\eps}^4\|\pt^{k+1} u\|_{|\al|-1}^{2}+C\|\pt ^k N\|_{|\al|-1}^{2}+\cubic.\,\,\,
	\end{eqnarray}
\end{lemma}

\begin{proof}
	Let $\al,\beta\in\mathbb{N}^3$ be multi-indices and $k$ be integers satisfying
	\[
	|\al|\geq1, \quad 0\leq k\leq s-1, \quad |\al|+k\le s, \quad |\beta|\leq |\al|-1.
	\]
	Applying $\pt ^k\pa_x^\beta$ to \eqref{EM30}$_2$ with \eqref{h}, one has
	\begin{eqnarray}\label{bEqula}
		&&h^\prime(n)\D \pt ^k N_\beta +\pt^k F_\beta\nonumber\\
		&=&-\eps^2\pt^{k+1} u_\beta-\pt^k u_\beta-\pt^k \pa_x^\beta(\eps u\times B+\eps^2((u\cdot\D) u))-R_1^{k,\beta},
	\end{eqnarray}
	with the remaining term defined as
	\[
	R_1^{k,\beta}=\pt^k \pa_x^\beta\Big(\int_0^1h''(\tilde{n}^\theta)\mathrm{d}\theta N \D n_e\Big)+\pt^k \pa_x^\beta(h^\prime(n)\D N)-h^\prime(n)\D \pt ^k N_\beta.
	\]
	Taking the inner product of \eqref{bEqula} with $\D \pt ^k N_\beta$ leads to the following
	\begin{eqnarray}\label{startptdn}
		&&\left<\D\pt^k N_\beta, h^\prime(n)\D \pt^k N_\beta\right>+\left<\pt^k F_\beta, \D \pt^k N_\beta\right>+\left<R_1^{k,\beta}, \D \pt^kN_\beta \right>\nonumber\\
		&=&-\left<\eps^2 \pt^k \pa_x^\beta((u\cdot\D)u)+\eps^2\pt^{k+1} u_\beta+\pt^k u_\beta+\eps\pt^k \pa_x^\beta(u\times B), \D \pt^k N_\beta\right>\nonumber\\
		&\leq& C\|\pt ^k u\|_{|\al|-1}^{2}+C{\eps}^4\|\pt^{k+1} u\|_{|\al|-1}^{2}+\cubic+\dfrac{h_1}{3}\|\D\pt^k N_\beta\|^2.
	\end{eqnarray}
	It is clear that
	\[
	\left<\D\pt^k N_\beta, h^\prime(n)\D \pt^k N_\beta\right>\geq h_1\|\D \pt^k N_\beta\|^2,\quad \left<\pt^k F_\beta, \D \pt^k N_\beta\right>=\|\pt^k N_\beta\|^2.
	\]
	For the term containing $R_1^{k,\beta}$ in \eqref{startptdn}, by \eqref{Moser00}, one has
	\begin{eqnarray}
		\left|\left<R_1^{k,\beta}, \D \pt^k N_\beta \right>\right|\leq \cubic+C\|\pt^k N\|_{|\al|-1}^2+\dfrac{h_1}{3}\|\D\pt^k N_\beta\|^2.\nonumber
	\end{eqnarray}
	Combining all these estimates and adding for all $0\le|\beta|\leq |\al|-1$ yield
	\begin{equation}\label{Nla}
		\|\pt ^k N\|_{|\al|}^{2} \le C\|(\pt ^k u, {\eps}^2\pt^{k+1} u,\pt ^k N)\|_{|\al|-1}^{2}+\cubic.
	\end{equation}
	
	We then need to bound $\pt^k F_\beta$. The equation \eqref{EM30}$_2$ naturally leads to
	\begin{equation*}
		\|\pt^k F_{\beta}\|^{2} \le C\|\pt^k u\|_{|\al|-1}^{2}+C{\eps}^4\|\pt^{k+1} u\|_{|\al|-1}^{2}+C\|\pt^k N\|_{|\beta|+1}^{2}+\cubic.
	\end{equation*}
	Summing the above for all  $|\beta|\le |\al|-1$ yields
	\begin{equation*}
		\|\pt ^k F\|_{|\al|-1}^{2} \le C\|\pt ^k u\|_{|\al|-1}^{2}+C{\eps}^4\|\pt^{k+1} u\|_{|\al|-1}^{2}+C\|\pt ^k N\|_{|\al|}^{2}+C\norm{U}_{s}^{2}\norm{W}_{s}.
	\end{equation*}
	which ends the proof of \eqref{NFla} by combining \eqref{Nla}.
\end{proof}

\subsection*{Proof of Theorem \ref{theorem1}} Based on Lemmas \ref{lemma3.1}-\ref{Lemma3.3}, now we can finish the proof of Theorem \ref{theorem1}. Adding $\al$ up to $|\al|$ in \eqref{ptinduction} and substituting \eqref{NFla} into the resulting equation yield
\begin{eqnarray}\label{ptspatial}
	&&\sum_{\substack{1\leq |\gamma|\leq |\al|\\k+|\gamma|\leq s}}\frac{\mathrm{d}}{\mathrm{d}t}\left(\left< D_0(\eps)A_0(n)\pt^k U_\gamma,\pt^k U_\gamma\right>+\|\pt^k F_\gamma\|^2+\|\pt^k G_\gamma\|^2\right)\nonumber\\
	&&+2n_1\|\pt^k u\|_{|\al|}^2+\|\pt^k N\|_{|\al|}^2+\|\pt^k F\|_{|\al|-1}^{2} \nonumber\\
	&\leq&  C\mu\norm{u}_{s}^2+C\|(\pt ^k u, {\eps}^2\pt^{k+1} u,\pt ^k N)\|_{|\al|-1}^{2}+\cubic,
\end{eqnarray}
where $\gamma\in\mathbb{N}^3$ satisfies $1\leq |\gamma|\leq |\al|$. Applying the induction argument on $|\al|$ in  \eqref{ptspatial}, we transfer the space derivatives order by order to the time derivatives. Combining the base cases \eqref{base1} and \eqref{base2}, one has for $|\al|\geq 1$ and $k+|\al|\leq s$,
\begin{eqnarray}
	&&\sum_{\substack{m+|\gamma|\leq s}}\frac{\mathrm{d}}{\mathrm{d}t}\left(\left< D_0(\eps)A_0(n)\pt^m U_\gamma,\pt^m U_\gamma\right>+\|\pt^m F_\gamma\|^2+\|\pt^m G_\gamma\|^2\right)\nonumber\\
	&&+2n_1\|\pt^k u\|_{|\al|}^2+\|\pt^k N\|_{|\al|}^2+\|\pt^k F\|_{|\al|-1}^{2}\leq C\mu\norm{u}_{s}^2+\cubic.
\end{eqnarray}
Adding the above for all $0\leq k\leq s-1$ and $1\leq |\al|\leq s$ with $k+|\al|\leq s$ and combining \eqref{base1} imply that there exists a constant $c_0>0$, such that
\begin{eqnarray}
	&&\norm{u}_s^2+\norm{N}_{s-1}^2+\norm{\D N}_{s-1}^2+\norm{F}_{s-1}^2\nonumber\\
	&\leq&c_0\mu\norm{u}_s^2+\cubic\nonumber\\
	&&-C\sum_{\substack{m+|\gamma|\leq s}}\frac{\mathrm{d}}{\mathrm{d}t}\left(\left< D_0(\eps)A_0(n)\pt^m U_\gamma,\pt^m U_\gamma\right>+\|\pt^m F_\gamma\|^2+\|\pt^m G_\gamma\|^2\right).\nonumber
\end{eqnarray}
By choosing $\mu>0$ such that
\begin{equation}\label{mu}
	c_0\mu<\dfrac{1}{2},
\end{equation}
one obtains
\begin{eqnarray}\label{middle4}
	&&\sum_{\substack{m+|\gamma|\leq s}}\frac{\mathrm{d}}{\mathrm{d}t}\left(\left< D_0(\eps)A_0(n)\pt^m U_\gamma,\pt^m U_\gamma\right>+\|\pt^m F_\gamma\|^2+\|\pt^m G_\gamma\|^2\right)\nonumber\\
	&&+\norm{u}_s^2+\norm{N}_{s-1}^2+\norm{\D N}_{s-1}^2+\norm{F}_{s-1}^2\leq\cubic.
\end{eqnarray}
Noticing the equivalence of $\|\pt^m  W_\gamma\|^2$ and $\left< D_0(\eps)A_0(n)\pt^m U_\gamma,\pt^m U_\gamma\right>+\|\pt^m F_\gamma\|^2+\|\pt^m G_\gamma\|^2$, integrating \eqref{middle4} over $[0,T]$ yields
\[
\norm{W(T)}_s^2+\int_0^T\left(\norm{u(t)}_s^2+\norm{(N,\D N, F)(t)}_{s-1}^2\right)\mathrm{d}t\leq C\|W(0)\|_s^2.
\]
It remains to bound $\pt^s N$ and $\D\times G$. Indeed,
\begin{eqnarray}\label{N-s}
	\|\pt^s N\|^2 &\leq& C\|\pt^{s-1} u\|_1^2\leq C\norm{u}_s^2, \\
	\norm{\D\times G}_{s-2}^2&\leq& \eps^2\norm{\pt F}_{s-2}^2 +\eps^2\norm{nu}_{s-2}^2\leq C\norm{U}_{s}^2. \nonumber
\end{eqnarray}
Hence the proof of Theorem \ref{theorem1} is complete. \hfill $\square$

\subsection*{Proof of Theorem \ref{theorem2}}
Once obtaining the uniform energy estimates \eqref{2-1}, we can prove Theorem \ref{theorem2}. Since sequences $\left\{(N^\eps,F^\eps,G^\eps)\right\}_{\eps>0}$ and $\{u^\eps\}_{\eps>0}$ are  bounded in $L^{\infty}\left(\mathbb{R}^{+} ; H^{s}\right)$ and $L^{2}\left(\mathbb{R}^{+}; H^{s}\right)$, respectively,  there exist functions $(\bar{N},\bar{u},\bar{F},\bar{G})$ such that
\begin{align*}
	\begin{split}
		(N^\eps,F^\eps,G^\eps)\stackrel{\ast}{\rightharpoonup} (\bar{N},\bar{F},\bar{G}),\quad & \text{weakly-*}\,\,\text{in}\,\, L^{\infty}\left(\mathbb{R}^+;H^s\right),\\
		u^\eps \rightharpoonup \bar{u},\quad &\text{weakly}\,\,\text{in}\,\, L^{2}\left(\mathbb{R}^+;H^s\right).
	\end{split}
\end{align*}
Besides, as $\eps\to 0$, in the sense of distributions, it holds
\begin{equation}\label{conv}
	\begin{array}{r}
		\eps^{2}\left[\pt u^\eps+(u^\eps\cdot\D) u^\eps\right]+\eps\left(u^\eps \times B^\eps\right) \rightarrow 0, \\
		\eps\left(\pt F^\eps-n^\eps u^\eps\right) \rightarrow 0, \quad \eps \pt G^\eps \rightarrow 0,
	\end{array}
\end{equation}
which allows us to pass to the limit for Maxwell equations in \eqref{EM30}, leading to
\begin{equation}
	\left\{ \begin{array}{l}
		\nabla\times \bar{G}=0,\quad\dive \bar{F} = -\bar{N},\\
		\nabla\times \bar{F}=0,\quad \dive \bar{G}=0.
	\end{array} \right.
\end{equation}
Here we learn that $\bar{G}$ is a constant vector and since $B_0^\eps-B_e$ converges weakly to $0$ in $H^s$, one obtains that $\bar{G}=0$. Let
\[
\bar{n}=\bar{N}+n_e,\quad \bar{E}=\bar{F}+E_e.
\]
Combining \eqref{PHT}$_2$--\eqref{PHT}$_3$, we arrive at
\begin{equation*}
\dive \bar{E} = b(x)-\bar{n}, \quad \D\times \bar{E}=0,
\end{equation*}
which implies that there exists a potential $\D\bar{\phi}$ satisfying
\begin{equation*}
    \Delta \bar{\phi}=b(x)-\bar{n}, \quad \bar{E}=\nabla \bar{\phi}.
\end{equation*}

In addition, from \eqref{2-1}, the sequence $\left\{\pt N^\eps\right\}_{\eps>0}$ is bounded in $L^{2}([0,T];H^{s-1})$. By the classical compactness theories (see \cite{Simon1987}), for any $0\le s'<s$, the sequence $\left\{N^\eps\right\}_{\eps>0}$ is relative compact in $\mathcal{C}([0,T] ;H^{s'})$, which yields that up to subsequences, $N^\eps$ converges strongly to $\bar{N}$. This together with \eqref{conv} enable us to pass to the limit of Euler equations in system \eqref{EM30} in the sense of distributions to obtain
\begin{equation*}
	\left\{ \begin{array}{l}
		{\pa _{t}}\bar{N} + \bar{u}\cdot\D\bar{N}+(\bar{N}+n_e)\cdot\dive\bar{u}+\bar{u}\cdot\D n_e = 0,\\
		\bar{u}=-\D (h\left(\bar{N}+n_e\right)-h(n_e))-\nabla\bar{F},
	\end{array} \right.
\end{equation*}
which further combining \eqref{PHT} yields \eqref{EM-drift-diffusion}--\eqref{baru} and thus ends the proof. \hfill $\square$

\section{Convergence rate for Euler-Maxwell system}
In this section, we establish the global-in-time error estimates between smooth solutions of the original systems \eqref{EM2} and the drift-diffusion system \eqref{EM-drift-diffusion} for periodic problems $\mathbb{K}=\mathbb{T}$. The proof is based on the results of the uniform estimate and the global-in-time convergence obtained in the previous section. For simplicity, we continue to drop the superscript $\eps$. In this section, we denote $(n,u,E,B)$ the global smooth solution to \eqref{EM30},  $(\bar{n},\bar{u},\bar{\phi})$  the global smooth solution to \eqref{EM-drift-diffusion}--\eqref{baru} and $(n_e,E_e,B_e)$ the stationary solution to \eqref{PHT}. For simplicity,  we denote
\[
\bar{N}=\bar{n}-n_e,\quad \bar{F}=\bar{E}-E_e.
\]

\subsection{Estimates for the drift-diffusion system} In order to study the error estimates, one first needs to study the limit equations. The next lemma concerns the estimates of solutions to drift-diffusion equations.

\begin{lemma}\label{lemmaA} Assume that $\|\bar{n}_{0}-n_e\|_{s}\leq \delta$, then $(\bar{N},\bar{u},\bar{F})$ satisfies
	\begin{eqnarray}\label{bar{rho}_es}
	&&	\|\bar{N}(t)\|_{s}^{2}+\|\pt\bar{N}(t)\|_{s-1}+\int_{0}^{t}\|\bar{N}(\tau)\|_{s+1}^{2}+\|\pt\bar{N}(\tau)\|_{s-1}^2 \mathrm{d}\tau \leq C\delta,\\
\label{bar{E}_es}
	&&	\|\bar{F}(t)\|_{s}^{2}+\|\pt \bar{F}(t)\|_{s-1}^{2}+\int_{0}^{t}\left(\|\bar{F}(\tau)\|_{s+1}^{2}+\|\pt \bar{F}(\tau)\|_{s}^{2}\right) \mathrm{d}\tau \leq C\delta, \\
&&\label{bar{v}_es}
		\|\bar{u}(t)\|_{s-1}^{2}+\|\pt \bar{u}(t)\|_{s-2}^{2}+\int_{0}^{t}\left(\|\bar{u}(\tau)\|_{s}^{2}+\|\pt\bar{u}(\tau)\|_{s-1}^{2}\right) \mathrm{d}\tau \leq C\delta.
	\end{eqnarray}
\end{lemma}
\begin{proof}
	By lower semi-continuity of $H^s$-norms and weak convergence \eqref{cov}, one has
	\begin{equation}\label{bar{n}_es}
		\|\bar{N}(t)\|_{s}^{2}+\|\bar{F}(t)\|_s^2+\int_{0}^{t}\left(\|\bar{N}(\tau)\|_{s}^{2}+\|\bar{u}(\tau)\|_s^2\right)\mathrm{d}\tau \leq C\delta, \quad \forall\, t>0.
	\end{equation}
	We denote $\bar{\Phi}=\bar{\phi}+ h(n_e)$, hence $\bar{F}=\D \bar{\Phi}$ and $\Delta \bar{\Phi}=-\bar{N}.$ By classical elliptic theories (See \cite{Gilbarg1983}), one obtains for any multi-index $\al\in\mathbb{N}^3$,
	\begin{equation}\label{barEbarn}
		\|\bar{F}\|_{|\al|}=\|\D\bar{\Phi}\|_{|\al|}\leq C\|\bar{N}\|_{|\al|},
	\end{equation}
	which implies the boundedness of $\bar{F}$ in $L^2(\R^+;H^s)$. Besides, relation \eqref{baru} implies
	\begin{equation}\label{DN}
		\bar{u}=-\bar{F}-[\D{h(\bar{n})}-\D{h(n_e)}]=-\bar{F}-h^\prime(\bar{n})\D\bar{N}-(h^\prime(\bar{n})-h^\prime(n_e))\D n_e,
	\end{equation}
	one thus obtains the boundedness of $\bar{u}$ in $L^\infty(\R^+;H^{s-1})$. Noticing the boundedness of $\bar{u},\bar{F},\bar{N}$ in $L^2(\mathbb{R}^+;H^s)$, one obtains from \eqref{DN} that
	\[
	\|\D\bar{N}\|_s\leq C\left(\|\bar{N}\|_s+\|\bar{F}\|_s+\|\bar{u}\|_s\right),
	\]
	which yields the boundedness of $\bar{F}$ in $L^2(\R^+;H^{s+1})$. Hence one concludes
	\begin{equation}\label{naka00}
		\|(\bar{N},\bar{F})(t)\|_s^2+\|\bar{u}(t)\|_{s-1}^2+\int_{0}^{t}\left(\|(\bar{N},\bar{F})(\tau)\|_{s+1}^2+\|\bar{u}(\tau)\|_{s}^2\right)\mathrm{d}\tau\leq C\delta.
	\end{equation}

    Next, we estimate $(\bar{N},\bar{u},\bar{F})$ with time derivatives. By lower semi-continuity of $H^s$-norms and weak convergence \eqref{cov}, one has
    \[
    \|(\pt\bar{N},\pt\bar{F})(t)\|_{s-1}^2+\int_0^t(\|\pt\bar{N}(\tau)\|_{s-1}^2+\|\pt\bar{u}(\tau)\|_{s-1}^2)\mathrm{d}\tau\leq C\delta.
    \]
    For  multi-indices $\beta\in\mathbb{N}^3$ satisfying $|\beta|\leq s$, applying $\pa_x^\beta\pt$ to both sides of $\Delta\bar{\Phi}=-\bar{N}$ and taking the inner product of the resulting equation with $ \pa_x^\beta\pt \bar{\Phi}$, one has
	\begin{eqnarray*}
		\|\pt\pa_x^{\beta}\bar{F}\|^2 \leq \left|\left<\pa_x^{\beta}
		\dive(\bar{n}\bar{u}),\pt\pa_x^{\beta}\Phi\right>\right| &=& \left|\left<\pa_x^{\beta}(\bar{n}\bar{u}),\pt\pa_x^{\beta}\bar{F}\right>\right|\\
		&\le& \dfrac{1}{2}\|\pt\pa_x^{\beta}\bar{F}\|^2 + C\|\pa_x^{\beta}(\bar{n}\bar{u})\|^2.
	\end{eqnarray*}
	Therefore, the boundedness of $\pt\bar{F}$ in $L^\infty(\R^+;H^{s-1})\cap L^2(\R^+;H^s)$ is obtained. It remains to estimate $\pt\bar{u}$. For multi-indices $\gamma\in\mathbb{N}^3$ satisfying $|\gamma|\le s-2$, applying $\pt\pa_x^\gamma$ to the relation for $\bar{u}$ \eqref{DN} yields
	\begin{eqnarray*}
		\|\pt\pa_x^\gamma\bar{u}\|
		&=& -\pa_x^\gamma\pt(h^\prime(\bar{n})\D\bar{N}-\pt\left((h^\prime(\bar{n})-h^\prime(n_e))\D n_e\right)+\bar{F})\nonumber\\
		 &\le& C\|\pt\bar{N}\|_{|\gamma|+1}+\|\pt \bar{F}\|_{|\gamma|},
	\end{eqnarray*}
	which implies $\pt\bar{u}\in L^\infty(\R^+;H^{s-2})$. Combining all these estimates and \eqref{naka00} reaches the desired estimates \eqref{bar{rho}_es}--\eqref{bar{v}_es}.
\end{proof}

\subsection{Energy estimates for error functions}
For readers' convenience, we state our strategies for obtaining the global-in-time convergence rates between smooth solutions to \eqref{EM2} and \eqref{EM-drift-diffusion}. Let $T>0$ be some positive time and $\eta>0$ be a sufficiently small constant, of which the value is determined in \eqref{eta}. In addition, we denote $\W=(\N,\eps\mathcal{U}^\top,\F^\top,\G^\top)^\top$ with
\begin{equation*}
	\left(\N, \mathcal{U}, \F,\G\right)=\left(n-\bar{n}, u-\bar{u}, E-\bar{E}, B-\bar{B}\right).
\end{equation*}
Then our proof outline is as follows.
\begin{itemize}
	\item We write the error system between \eqref{EM2} and \eqref{EM-drift-diffusion} into an anti-symmetric form, of which we are able to take advantage to establish estimates for $(\N,\eps\U)$ in  $L^\infty(\R^+;H^{s-1})$ as well as $\U$ in $L^2(\R^+;H^{s-1})$(see Lemma \ref{lemma4.3-anti}).
	\item Stream function technique is applied to establish estimates for $\F$ and $\G$ in $L^\infty(\R^+;H^{s-1})$(see Lemmas \ref{lemma6.2}-\ref{lemma4.5}) as well as $(\N,\F)$ in $L^2(\R^+;H^{s-1})$ (see Lemma \ref{lemma4.6}) together with $\D\G$ in $L^2(\R^+;H^{s-2})$ by Maxwell's equations.
	\item Finally, an induction argument is carried out on the order of the derivatives at the end of the section.
\end{itemize}

In order to carry out the proof, we first reveal the anti-symmetric structure of the error system. Noticing \eqref{EM2} and \eqref{EM-limit}, one has
\begin{equation}\label{global-limit}
	\begin{cases}
		\pt\N + n \dive\U + \D\N\cdot u + \D\bar{n}\cdot\U = - \N\dive\bar{u} ,\\
		\eps^2\pt\U+ \eps^2 ((u \cdot \D)) \U+h'(n)\D\N +\N\D  h'(\bar{n})\nonumber\\
		\qquad\qquad\qquad\qquad\qquad=-\F-\U-\eps u \times B - r(\bar{n},\N) - \eps^2\pt\bar{u} -\eps^2 ((u\cdot\D)\bar{u}),
	\end{cases}
\end{equation}
which can be written as the following first-order quasi-linear system:
\begin{equation}\label{Sym-limit}
	D_0(\eps)\pt \V+\sum_{j=1}^{3} A_{j}(n,u) \pa_{x_{j}} \V+\hat{L}(\bar{n}) \V= \widetilde{f},
\end{equation}
with $\V=(\N,\U^\top)^\top, D_0(\eps)=\text{diag}(1,\eps^2\mathbb{I}_{3})$. For $j=1,2,3$,
\begin{eqnarray*}
A_{j}(n, u)&=&\left(\begin{array}{cc}u_{j} & n e_{j}^\top \\
	h'(n) e_{j} & \eps^2 u_{j} \mathbb{I}_{3}
\end{array}\right),\quad
\hat{L}(\bar{n})=\left(
\begin{matrix}
	0& (\D \bar{n})^\top\\
	\D h^{\prime}(\bar{n})&0\\
\end{matrix}
\right)\nonumber\\
\widetilde{f}&=&\left(\begin{matrix}
	- \N\dive\bar{u}\\
	-\F-\U-\eps u \times B - r(\bar{n},\N) - \eps^2\pt\bar{u} -\eps^2 (u\cdot\D)\bar{u}
\end{matrix}\right).
\end{eqnarray*}
The remaining term $r$ is defined as
\[
r(\bar{n},\N)=\left(h^{\prime}(n)-h^{\prime}(\bar{n})-h^{\prime \prime}(\bar{n}) \N\right) \D \bar{n}=O\left(\N^{2}\right).
\]
For any multi-index $\al\in\mathbb{N}^3$, we denote for simplicity:
\begin{equation*}
	\V_\al=\pa_x^\al \V, \quad \left(\N_{\al}, \mathcal{U}_{\al}, \F_\al, \G_{\al}\right)=\left(\pa_{x}^{\al} \N, \pa_{x}^{\al} \mathcal{U}, \pa_{x}^{\al} \F, \pa_{x}^{\al} \G\right).
\end{equation*}
as well as the following functionals
\[
\E_T=\sup_{t\in[0,T]}\|\W(t)\|_{s-1}^2, \quad \mD_T=\int_0^T\|(\V,\F)(t)\|_{s-1}^2\mathrm{d}t+\int_0^T\|\D\times\G(t)\|_{s-2}^2\mathrm{d}t.
\]

\subsubsection{Application of the anti-symmetric structure}

\begin{lemma}\label{lemma4.3-anti}
	For all $\al\in\mathbb{N}^3$ with $1\leq |\al|\leq s-1$, it holds
	\begin{align}\label{NU-limit1}
		&\|(\N_\al,\eps\U_\al)(T)\|^2 + 2n_1\int_0^T\|\U_\al(t)\|^2 \mathrm{d}t   +2\int_0^T\left<n\F_\al,\U_\al\right>\mathrm{d}t  \nonumber\\
		\leq &C\int_0^T\|\N(t)\|_{|\al|}^2 \mathrm{d}t+C\eps^{2p_1}+C(\delta+\eta)(\E_T+\mD_T).
	\end{align}
Especially, when $|\al|=0$, it holds
\begin{eqnarray}\label{zeror}
	&&\|(\N,\eps\U)(T)\|^2 + 2n_1\int_0^T\|\U(t)\|^2 \mathrm{d}t   +2\int_0^T\left<n\F,\U\right>\mathrm{d}t\nonumber\\
	&\leq& C\eps^{2p_1}+C(\delta+\eta)(\E_T+\mD_T).
\end{eqnarray}
\end{lemma}

\begin{proof} For multi-indies $\al\in\mathbb{N}^3$ with $|\al|\leq s$, applying $\pa_{x}^\al$ to \eqref{Sym-limit}, one obtains
	\begin{equation}\label{sym-limit-al}
		D_0(\eps)\pt \V_\al+\sum_{j=1}^{3} A_{j}(n,u) \pa_{x_{j}} \V_\al+\hat{L}(\bar{n}) \V_\al= \pa_{x}^\al\widetilde{f} + \hat{g}^\al,
	\end{equation}
with the commutator defined as
	\[
	\hat{g}^\al=\sum_{j=1}^{3} A_{j}(n, u) \pa_{x_{j}} \V_\al - \pa_x^\al\left(\sum_{j=1}^{3} A_{j}(n, u) \pa_{x_{j}} \V\right)+\hat{L}(\bar{n}) \V_\al -\pa_x^\al\left(\hat{L}(\bar{n}) \V\right) .
	\]
Taking the inner product of \eqref{sym-limit-al} with $2A_0(n)\V_\al$ with $A_0(n)$ defined in \eqref{symmetrizer} yields
	\begin{eqnarray}\label{startregular}
		\frac{\mathrm{d}}{\mathrm{d} t}\left< D_0(\eps)A_{0}(n) \V_\al, \V_\al\right>&=&\left<D_0(\eps)\pt A_0(n)\V_\al, \V_\al \right>+\left<\hat{B}(\V,\D \V)\V_\al,\V_\al\right>\nonumber\\
		&&+2\left< A_{0}(n)  \hat{g}^\al, \V_\al\right>+2\left< A_{0}(n) \pa_{x}^\al\widetilde{f},\V_\al\right>:=\sum_{j=1}^4{J_j^\al},\quad\quad
	\end{eqnarray}
with the natural correspondence of $\{J_j^\al\}_{j=1}^4$, and matrix $\hat{B}(\V, \D \V)$  defined as
	\begin{eqnarray}
		\hat{B}(\V,\D \V)&=&\sum_{j=1}^{3} \pa_{x_{j}}(A_0(n) A_{j}(n, u))-2 A_{0}(n) \hat{L}(\bar{n})\nonumber\\
		&=&\left(\begin{array}{cc}\dive\left(h^{\prime}(n) u\right) & \left(\D P^{\prime}(n)-2 h^{\prime}(n) \D \bar{n}\right)^\top \\ \D P^{\prime}(n)-2 n \D h^{\prime}(\bar{n})& \dive(\eps^2  nu) \mathbb{I}_{3}\end{array}\right),\nonumber
	\end{eqnarray}
	which is anti-symmetric at $n=\bar{n}$ by using the similar anti-symmetry technique stated in \eqref{antisymm}.	Similar to \eqref{It1}, one has for $J_1^\al$,
	\begin{equation}\label{J_1}
		\int_0^T|J_1^\al|\mathrm{d}t\leq C\sup_{t\in[0,T]}\norm{\pt N(t)}_{s-1}\int_0^T\|\V(t)\|_{s-1}^2\mathrm{d}t\leq  C\delta\mD_T.
	\end{equation}
For $J_2^\al$, similar to \eqref{It2}, one has
	\begin{eqnarray}\label{J_2}
		\int_0^T|J_2^\al|\mathrm{d}t&\leq& C\int_0^T\|u\|_s(\|\N_\al\|^2+\|\eps\U_\al\|^2)\mathrm{d}t+C\int_0^T\|\N\|_s\|\N_\al\|\|\U_\al\|\mathrm{d}t\nonumber\\
		&\leq& C\int_0^T\|(u,N,\bar{N})(t)\|_{s}^2\|\V(t)\|_{s-1}^2\mathrm{d}t\leq C\delta\mD_T.
	\end{eqnarray}
For $J_3^\al$, by \eqref{Moser00}, it holds
	\begin{equation*}
		\begin{aligned}
			\dfrac{1}{2}|J_3^\al| \leq
			& |\left<u\cdot\D\N_\al-\pa_x^\al(u\cdot\D\N),h^\prime(n)\N_\al\right>|+|\left<n\dive{\U_\al}-\pa_x^\al(n\dive{\U}),h^\prime(n)\N_\al\right>|\\
			&+|\left<h'(n)\D\N_\al-\pa_x^\al(h'(n)\D\N),n\U_\al\right>|\nonumber\\
			&+\eps^2|\left< ((u\cdot\D)\U_\al)-\pa_x^\al((u\cdot\D)\U),n\U_\al\right>|\\
			&+|\left<(\D\bar{n})\cdot\U_\al-\pa_x^\al((\D\bar{n})\cdot\U),h^\prime(n)\N_\al\right>|\nonumber\\
			&+|\left<\D h'(\bar{n})\N_\al-\pa_x^\al(\D h'(\bar{n})\N),n\U_\al\right>|\\\leq
			& C\|u\|_{s}\|\N\|_{|\al|}^2 + C\|\U\|_{|\al|}\|\N\|_{|\al|} + C\eps^2\|u\|_s\|\U\|_{|\al|}^2.
		\end{aligned}
	\end{equation*}
	Integrating the above over $[0,T]$, one obtains
	\begin{equation}\label{J_3}
		\begin{split}
			\int_0^T |J_3^\al| \mathrm{d}t   \leq&\,\,\, C\int_0^T\|\N(t)\|_{|\al|}^2\mathrm{d}t  + C(\delta+\eta)\mD_T.
		\end{split}
	\end{equation}
Especially, for $|\al|=0$, one has $J_3^\al=0$. For $J_4^\al$, similar to \eqref{It4naka}, one obtains by using \eqref{3.1} and  inequalities \eqref{Moser00}--\eqref{Mosertt} that
	\begin{align}\label{J_4}
			&\int_0^TJ_4^\al\mathrm{d}t+		2n_1\int_0^T\|\U_\al(t)\|^2\mathrm{d}t-2\int_0^T\left<n\F_\al,\U_\al\right>\mathrm{d}t\\
			\leq & C\int_0^T\|\bar{u}(t)\|_s\|\N(t)\|_{s-1}^2\mathrm{d}t+ C\eps^2\int_0^T\left<\pt\pa_x^\al\bar{u}+\pa_x^\al((u\cdot\D)\bar{u}),n\U_\al\right>\mathrm{d}t \nonumber\\
			&+ C\int_0^T\|\N(t)\|_{s}\|\N(t)\|_{|\al|}\|\U(t)\|_{|\al|}\mathrm{d}t+ C\int_0^T\left<\eps n\pa_x^\al(u\times B),\U_\al\right>\mathrm{d}t\nonumber\\
			\leq& C\eps^2+C(\delta+\eta)(\E_T+\mD_T).
	\end{align}
	Noticing the equivalence of $\left< D_0(\eps)A_{0}(n) \V_\al, \V_\al\right>$ and $\|\pa_x^\al \W\|^2$, integrating \eqref{startregular} over $[0,T]$ and combining \eqref{J_1}--\eqref{J_4} yield \eqref{NU-limit1}.
\end{proof}

\subsubsection{Applications of the stream function technique} We first find a proper stream function. We consider the error for mass equations \eqref{EM2} and \eqref{EM-limit} as our conservative equation:
\begin{equation}\label{masserr}
	\pt \N+\dive\left(n u-\bar{n} \bar{u}\right)=0.
\end{equation}
Notice from \eqref{EM2} and \eqref{EM-drift-diffusion} that
\begin{equation*}
	\dive \F=\dive E-\dive \bar{E}=-\N,
\end{equation*}
which implies that $\F$ is a natural candidate for our stream function. However, we need to recover $\pt\bar{E}$ due to its loss of information when $\eps\to 0$. From \eqref{EM-drift-diffusion} and the mass equation in \eqref{EM-limit}, one obtains
\begin{equation*}
	\dive \pt \bar{E}=-\pt \bar{n}=\dive(\bar{n} \bar{u}),
\end{equation*}
which implies that there exists a unique function $\bar{H}$ such that
\begin{equation}\label{H1}
\pt  \bar{E}-\bar{n} \bar{u}=\D \times \bar{H}, \quad \dive \bar{H}=0, \quad m_{\bar{H}}(t)=0.
\end{equation}
Based on this, one obtains that the stream function $\F$ associated with \eqref{masserr} satisfies
\[
\dive \F=-\N, \qquad \pt\F=\pt E-\pt\bar{E}=\left(n u-\bar{n} \bar{u}\right)+\frac{1}{\eps} \D \times \G-\D \times \bar{H}.
\]

In this subsection, we tend to use the stream function technique to obtain the global error estimates for $\N $ and $\F$.
First, we give the estimates for $\bar{H}$.

\begin{lemma}\label{lemma6.2}
	The solution $\bar{H}$ to \eqref{H1} satisfies
	\begin{equation}\label{stream}
		\bar{H} \in L^{\infty}\left(\mathbb{R}^{+} ; H^{s}\right),\quad \pt \bar{H} \in L^{2}\left(\mathbb{R}^{+} ; H^{s}\right).
	\end{equation}
\end{lemma}
\begin{proof} Applying $\D\times$ in \eqref{H1}$_1$  yields
	\begin{equation}\label{deltaH}
		\Delta \bar{H}=\D \times(\bar{n} \bar{u}), \quad m_{\bar{H}}(t)=0.
	\end{equation}
	Consequently, for multi-indices $\al\in\mathbb{N}^d$ with $|\al|\leq s-1$, classical energy estimates together with the Young inequality yield
	\[
\|\pa_x^\al\D\bar{H}\|^2=\left<\pa_x^\al\bar{H},\pa_x^\al\Delta\bar{H}\right>=\left<\D\times\pa_x^\al\bar{H},\pa_x^\al(\bar{n}\bar{u})\right>\leq \frac{1}{2}\|\D\pa_x^\al\bar{H}\|^2+C\|\pa_x^\al(\bar{n}\bar{u})\|^2,
	\]
	which implies $\D \bar{H} \in L^{\infty}\left(\mathbb{R}^{+} ; H^{s-1}\right)$. In addition, it holds
	\begin{equation*}
		\pt(\bar{n} \bar{u})=\left(\pt \bar{n}\right) \bar{u}+\bar{n} \pt \bar{u}=-\dive(\bar{n} \bar{u}) \bar{u}+\bar{n} \pt \bar{u},
	\end{equation*}
	which yields $\pt(\bar{n} \bar{u}) \in L^{2}\left(\mathbb{R}^{+} ; H^{s-1}\right)$.
	Taking the time derivative to \eqref{deltaH} leads to
	\begin{equation*}
		\Delta \pt \bar{H}=\D \times \pt(\bar{n} \bar{u}).
	\end{equation*}
	The proof is complete by classical elliptic theories and the Poincar\'e inequality.
\end{proof}

The next lemma is a direct application of the stream function technique.

\begin{lemma}\label{lemma4.4}
	For all $\al \in \mathbb{N}^3 $ with $|\al| \le s-1$, it holds
	\begin{eqnarray}\label{FG-limit}
		\|\F_\al(T)\|^{2}+\|\G_\al(T)\|^{2}\leq C\eps^{2 p_{1}} + 2\int_0^T \left< \pa_{x}^{\al}(n u-\bar{n} \bar{u}),\F_\al\right> \mathrm{d}t   +C\eta \mD_T.
	\end{eqnarray}
	
\end{lemma}
\begin{proof} For  multi-indices $\al\in\mathbb{N}^3$ with $|\al| \le s-1$, applying $\pa_x^\al$ to \eqref{H1}$_2$ yields
	\begin{equation*}
		\pa_{x}^{\al}(n u-\bar{n} \bar{u})=\pt \F_\al-\frac{1}{\eps} \D \times \G_{\al}+\D \times \pa_{x}^{\al} \bar{H}.
	\end{equation*}
	Taking the inner product of the above equation with $\F_\al$, integrating the resulting equation over $[0,T]$, combining Lemma \ref{lemma6.2} and using Young's inequality, one has
	\begin{eqnarray}
		&&\int_0^T \left< \pa_x^\al(n u-\bar{n} \bar{u}),\F_\al\right> \mathrm{d}t  \nonumber\\ &=&\int_0^T\left(\frac{1}{2} \frac{\mathrm{d}}{\mathrm{d}t}\|\F_\al(t)\|^{2}-\left<\D \times \F_\al, \frac{1}{\eps} \G_{\al}- \pa_{x}^{\al}\bar{H}\right>\right) \mathrm{d}t  \nonumber\\
		&=& \int_0^T \frac{1}{2} \frac{\mathrm{d}}{\mathrm{d}t}\left(\|\F_\al(t)\|^{2} + \|\G_{\al}(t)\|^2 \right)\mathrm{d}t -\int_0^T\left(\frac{\mathrm{d}}{\mathrm{d}t}\left<\eps \G_{\al}, \pa_{x}^{\al} \bar{H}\right>-\left<\eps \G_{\al}, \pt \pa_{x}^{\al} \bar{H}\right>\right) \mathrm{d}t\nonumber\\
		&\ge & \frac{1}{2}\|(\F_\al,\G_\al)(T)\|^{2}-C \eps^{2 p_{1}}-\int_0^T \frac{\mathrm{d}}{\mathrm{d}t}\left<\eps \G_{\al}, \pa_{x}^{\al} \bar{H}\right> \mathrm{d}t+\int_0^T\left<\eps \G_{\al}, \pt \pa_{x}^{\al} \bar{H}\right> \mathrm{d}t\nonumber\\
		&\geq& \frac{1}{4}\|(\F_\al,\G_\al)(T)\|^{2}-C \eps^{2 p_{1}}+\int_0^T\left<\eps \G_{\al}, \pt \pa_{x}^{\al} \bar{H}\right> \mathrm{d}t.\nonumber
	\end{eqnarray}
	Now it suffices to prove
	\begin{equation}\label{lemma6.6suffice}
		\left|\int_0^T\left<\eps \G_{\al}, \pt \pa_{x}^{\al} \bar{H}\right> \mathrm{d}t\right| \leq C \eps^{2}+\eta \int_0^T\|\D\times \G(t)\|_{s-2}^{2} \mathrm{d}t.
	\end{equation}
	Actually, for $1\leq|\al|\leq s-1$, \eqref{lemma6.6suffice} is obvious by noticing Lemma \ref{lemma6.2} and Theorem \ref{theorem1}. When $\al=0$, since $\dive\G=0$, there exists a unique function $\chi^\eps$ such that
	\begin{equation*}
			\D \times \chi^{\eps}=\G, \quad \dive \chi^{\eps}=0, \quad m_{\chi^\eps}(t)=0,
	\end{equation*}
	which implies the Poisson equation $\Delta\chi^\eps=-\D\times \G$. By classical elliptic theories, $\|\chi^\eps\|$ is bounded by $\|\D \times \G\|$, and consequently,
	\begin{equation*}
		\begin{aligned}
			\left<\eps \G,\pt  \bar{H}\right>
			&=\left<\eps\D\times\chi^\eps,\pt \bar{H}\right>
			\\&=\left<\eps\chi^\eps,\D\times\pt \bar{H}\right>\leq \eta\|\D\times \G\|^2+C\eps^2\|\D\times \pt  \bar{H}\|^2,
		\end{aligned}
	\end{equation*}
	which implies \eqref{lemma6.6suffice} for the case $\al=0$ and thus ends the proof.
\end{proof}

Combining \eqref{NU-limit1} and \eqref{FG-limit}, one has for $1\leq|\al|\leq s-1$,
\begin{eqnarray}\label{middleestim}
	&&\|(\N_\al, \eps\U_\al,\F_\al,\G_\al)(T)\|^2 + 2n_1\int_0^T\|\U_\al(t)\|^2 \mathrm{d}t+K^\al\nonumber\\
	&\leq&C\int_0^T\|\N(t)\|_{|\al|}^2 \mathrm{d}t+C\eps^{2p_1}+C(\delta+\eta)(\E_T+\mD_T).
\end{eqnarray}
Especially, for $|\al|=0$, one obtains from \eqref{zeror} and \eqref{FG-limit} that
\begin{eqnarray}\label{001}
	&&\|(\N, \eps\U,\F,\G)(T)\|^2 + 2n_1\int_0^T\|\U(t)\|^2 \mathrm{d}t+K^0\nonumber\\
	&\leq&C\eps^{2p_1}+C(\delta+\eta)(\E_T+\mD_T).
\end{eqnarray}
Here, $K^\al$ is defined as
\[
K^\al:=-2\int_0^T\left<F_\al,\pa_x^\al(nu-\bar{n}\bar{u})-n\U_\al\right>\mathrm{d}t,
\]
of which the estimate is given in the following lemma.
\begin{lemma}\label{lemma4.5} For $1\leq |\al|\leq s-1$, one has
	\begin{eqnarray}\label{middle11}
		|K^\al|\leq C(\delta+\eta)\mD_T+C\int_0^T\|\F(t)\|_{|\al|-1}^2\mathrm{d}t.\,\,\,\,\,\,\,\,\,
	\end{eqnarray}
Moreover, for $|\al|=0$, one obtains
\begin{eqnarray}\label{002}
\left|K^0\right|\leq C(\delta+\eta)\mD_T.
\end{eqnarray}
\end{lemma}
\begin{proof} Notice the explicit expression for $K^\al$, one obtains
	\begin{eqnarray}\label{middle0}
		K^\al &=&\int_0^T\left<\pa_{x}^{\al}(\N\U)-\N\U_\al, \F_\al\right>\mathrm{d}t+\int_0^T\left<\pa_{x}^{\al}(\bar{u}\N), \F_\al\right>\mathrm{d}t\nonumber\\
		&&+\int_0^T\left<\pa_{x}^{\al}(\bar{n}\U)-\bar{n}\U_\al, \F_\al\right>\mathrm{d}t.
	\end{eqnarray}
	Since $\N$ and $\bar{u}$ are small in $L^\infty(\mathbb{R}^+;H^{s-1})$, one has
	\begin{eqnarray}
		&&\left|\int_0^T\left<\pa_{x}^{\al}(\N\U)-\N\U_\al, \F_\al\right>\mathrm{d}t+\int_0^T\left<\pa_{x}^{\al}(\bar{u}\N), \F_\al\right>\mathrm{d}t\right|\leq C\delta\mD_T.
	\end{eqnarray}
	As to the last term on the right hand side of \eqref{middle0}, similar to the treatment in \eqref{treatF} and noticing that $\bar{n}\in L^\infty(\R^+;H^{s+1})$, one has
	\begin{eqnarray}
		\left|\int_0^T\left<\pa_{x}^{\al}(\bar{n}\U)-\bar{n}\U_\al, \F_\al\right>\mathrm{d}t\right|\leq \eta\int_0^T\|\V(t)\|_{s-1}^2\mathrm{d}t+C\int_0^T\|\F(t)\|_{|\al|-1}^2\mathrm{d}t,\nonumber
	\end{eqnarray}
	which yields \eqref{middle11}. For the case of $|\al|=0$, direct calculations give \eqref{002}.
\end{proof}

The next lemma gives dissipative estimates for $\N$ and $\F$.

\begin{lemma}\label{lemma4.6}
	For all $\al \in \mathbb{N}^3 $ with $1\le |\al| \le s$, it holds
	\begin{equation}\label{mathcalNF}
		\int_0^T\left(\|\F(t)\|_{|\al|-1}^2 + \|\N(t)\|_{|\al|}^2 \right)  \mathrm{d}t
		\le C \int_0^T\|\V(t)\|_{|\al|-1}^2 \mathrm{d}t+ C\eps^2.
	\end{equation}
	Especially, it holds
	\begin{eqnarray}\label{base00}
	    \int_0^T\|(\N,\F)(t)\|^2\mathrm{d}t&\leq& C\eps^2+C\int_0^T\|\U(t)\|^2\mathrm{d}t.
	\end{eqnarray}
\end{lemma}

\begin{proof}
	Subtracting the equation for $\bar{u}$ in \eqref{baru} from \eqref{EM2}$_2$ leads to
	\begin{equation}\label{start11}
		\eps^2\left(\pt u + (u \cdot \D u)\right)+\left( \D h(n) -\D h(\bar{n}) \right)=-\F -\eps u \times B -\U.
	\end{equation}
	Let multi-indices $\al,\beta\in\mathbb{N}^3$ with $1\leq |\al|\leq s$ and $|\beta|\leq |\al|-1$. Applying $\F_\beta\pa_{x}^\beta$ to the above equation and integrating over $[0,T]$ yield
	\begin{eqnarray}\label{mathcalNF_eq}
		&&\int_0^T\|\F_\beta(t)\|^2\mathrm{d}t+\int_0^T\left<h^\prime(\hat{n})\N_\beta, \N_\beta\right>\mathrm{d}t\nonumber\\
		&\leq& -\int_0^T\left<F_\beta,\eps^2\pt{u}_\beta+\eps^2\pa_{x}^\beta((u\cdot \D) u)+\eps \pa_{x}^\beta(u\times B)\right>\mathrm{d}t-\int_0^T\left< \U_\beta,\F_\beta \right>\mathrm{d}t\nonumber\\
		&&+\int_0^T\left<\pa_x^\beta(h^\prime(\hat{n})\N)-h^\prime(\hat{n})\N_\beta), \N_\beta\right>\mathrm{d}t
	\end{eqnarray}
	where $\hat{n}$ is between $n$ and $\bar{n}$. One obtains that
	\[
	\int_0^T\left<h^\prime(\hat{n})\N_\beta, \N_\beta\right>\mathrm{d}t\geq h_1\int_0^T\|\N_\beta(t)\|^2\mathrm{d}t.
	\]
	Direct calculation shows that
	\begin{eqnarray}
	&&\int_0^T\left<\F_\beta,\eps^2\pt{u}_\beta+\eps^2\pa_{x}^\beta((u\cdot \D) u)+\eps \pa_{x}^\beta(u\times B)\right>\mathrm{d}t+\int_0^T\left< \U_\beta,\F_\beta \right>\mathrm{d}t\nonumber\\
	&\leq& \dfrac{1}{3}\int_0^T \|\F_\beta(t)\|^2\mathrm{d}t+C\eps^2+C\int_0^T\|\U_\beta(t)\|^2 \mathrm{d}t.\nonumber
	\end{eqnarray}
	In addition, by the Moser-type calculus inequalities, for $|\beta|\geq 1$,
	\[
	\left|\int_0^T\left<\pa_x^\beta(h^\prime(\hat{n})\N)-h^\prime(\hat{n})\N_\beta), \N_\beta\right>\mathrm{d}t\right|\leq \dfrac{h_1}{2}\int_0^T\|\N_\beta(t)\|^2\mathrm{d}t+C\int_0^T\|\N(t)\|_{|\beta|}^2\mathrm{d}t,
	\]
	and the above estimate has no need to be carried out for $|\beta|=0$. Adding \eqref{mathcalNF_eq} for all $\beta$ up to $ |\al|-1$ and combining all these estimates above yield that
	\begin{eqnarray}\label{nakar1}
		\int_0^T\|(\N,\F)(t)\|_{|\al|-1}^2\mathrm{d}t&\leq& C\eps^2+C\int_0^T \|\V(t)\|_{|\al|-1}^2\mathrm{d}t.
	\end{eqnarray}
Especially, when $|\al|=0$, one obtains \eqref{base00}.

	Next, for $|\beta|\leq |\al|-1$, similarly as Lemma \ref{Lemma3.3}, applying $\D\N_\beta\pa_x^\beta$ to \eqref{start11} and integrating the resulting equation over $[0,T]$ yield
	\begin{eqnarray}
		&&\int_0^T\left<h^\prime(\hat{n})\D\N_\beta, \D\N_\beta\right>\mathrm{d}t+\int_0^T\left<F_\beta, \D \N_\beta\right>\mathrm{d}t\nonumber\\
		&\leq& -\int_0^T\left<\D\N_\beta,\eps^2\pt{u}_\beta+\eps^2\pa_{x}^\beta((u\cdot \D) u)+\eps \pa_{x}^\beta(u\times B)\right>\mathrm{d}t-\int_0^T\left< \U_\beta,\D\N_\beta \right>\mathrm{d}t\nonumber\\
		&&+\int_0^T\left<\pa_x^\beta\D(h^\prime(\hat{n})\N)-h^\prime(\hat{n})\D\N_\beta), \D\N_\beta\right>\mathrm{d}t\nonumber\\
		&\leq& \dfrac{h_1}{2}\int_0^T\|\D\N_\beta(t)\|^2\mathrm{d}t+ C\int_0^T \|\U_\beta(t)\|^2\mathrm{d}t+C\int_0^T\|\N(t)\|_{|\beta|}^2\mathrm{d}t+C\eps^2,\nonumber
	\end{eqnarray}
	where one has
	\[
	\left<F_\beta, \D \N_\beta\right>=\int_0^T\|\N_\beta(t)\|^2\mathrm{d}t,\quad \left<h^\prime(\hat{n})\D\N_\beta, \D\N_\beta\right> \geq h_1\|\D\N_\beta(t)\|^2.
	\]
	Consequently, one obtains that
	\[
	\int_0^T\|\D\N_\beta(t)\|^2\mathrm{d}t \leq C\eps^2+C\int_0^T\|\N(t)\|_{|\beta|}^2\mathrm{d}t+C\int_0^T \|\U_\beta(t)\|^2\mathrm{d}t.
	\]
	Adding the above inequality for all $|\beta|\leq |\al|-1$ yields
	\begin{equation*}
		\int_0^T\|\N(t)\|_{|\al|}^2\mathrm{d}t\leq C\eps^2+C\int_0^T \|\N(t)\|_{|\al|-1}^2\mathrm{d}t+C\int_0^T\|\U(t)\|_{|\al|-1}^2 \mathrm{d}t,
	\end{equation*}
	in which further combining \eqref{nakar1} and \eqref{base00} yields \eqref{mathcalNF}.
\end{proof}

\subsection*{Proof of Theorem \ref{convergence_rate}}
Substituting \eqref{middle11} and \eqref{mathcalNF} into \eqref{middleestim} for $1\leq |\al|\leq s$, one obtains
\begin{eqnarray}\label{inductionstart}
	&&\|(\N_\al, \eps\U_\al,\F_\al,\G_\al)(T)\|^2 + \int_0^T\|\U_\al(t)\|^2 \mathrm{d}t+\int_0^T\left(\|\F(t)\|_{|\al|-1}^2 + \|\N(t)\|_{|\al|}^2 \right)  \mathrm{d}t \nonumber\\
	&\leq&C\eps^{2p_1}+C(\delta+\eta)(\E_T+\mD_T)+C\int_0^T\|\V(t)\|_{|\al|-1}^2\text{d}t.
\end{eqnarray}
Especially, when $|\al|=0$, combining \eqref{001}, \eqref{002} and \eqref{base00} yields
\begin{equation}\label{inductionbase}
  \|(\N, \eps\U,\F,\G)(T)\|^2 + \int_0^T\|(\N,\U,\F)(t)\|^2 \mathrm{d}t	\leq C\eps^{2p_1}+C(\delta+\eta)(\E_T+\mD_T).
\end{equation}
Applying the induction argument on $|\al|$ in \eqref{inductionstart} and combining \eqref{inductionbase}, one obtains
\begin{eqnarray}
	&&\|(\N, \eps\U,\F,\G)(T)\|_{|\al|}^2 + \int_0^T\left(\|\U(t)\|_{|\al|}^2+\|\F(t)\|_{|\al|-1}^2 + \|\N(t)\|_{|\al|}^2 \right)  \mathrm{d}t \nonumber\\
	&\leq&C\eps^{2p_1}+C(\delta+\eta)(\E_T+\mD_T). \nonumber
\end{eqnarray}
Adding the above for all $ |\al|\leq s-1$, combining \eqref{inductionbase} and noticing \eqref{mathcalNF} for the case $|\al|=s$, one obtains
\begin{eqnarray}\label{finbefore}
	&&\|(\N, \eps\U,\F,\G)(T)\|_{s-1}^2 +\int_0^T\left(\|\U(t)\|_{s-1}^2+\|\F(t)\|_{s-2}^2 + \|\N(t)\|_{s-1}^2 \right)  \mathrm{d}t \nonumber\\
	&\leq&C\eps^{2p_1}+C(\delta+\eta)(\E_T+\mD_T).
\end{eqnarray}
In addition, the error for the Maxwell equations are of the form
\[
\D\times \G=\eps\pt F-\eps(nu), \quad \eps\pt \G+\D\times \F=0.
\]
By using \eqref{2-1} and \eqref{bar{E}_es}, one obtains directly  that
\[
\int_0^T\|\D\times \G(t)\|_{s-2}^2\mathrm{d}t\leq C\eps^2.
\]
Consequently, estimate \eqref{finbefore} implies that there exists a constant $c_1>0$ such that
\[
\E_T+\mD_T\leq C\eps^{2p_1}+c_1(\delta+\eta)(\E_T+\mD_T).
\]
Then one may choose $\delta$ and $\eta$ sufficiently small such that
\begin{equation}\label{eta}
    c_1(\delta+\eta)\leq \frac{1}{2},
\end{equation}
and thus the proof is complete. \hfill $\square$

\section{Applications for Euler-Poisson system}
In this section, we apply our methods to Euler-Poisson system. We first give the global convergence in zero-relaxation limit of the system \eqref{EPmain}, and then deduce the global error estimates. In the following, we drop the superscript of $\eps$. For simplicity, we still adopt the similar notations
\begin{align*}
    N=&\,\,n-n_e,\quad \Phi=\phi-\phi_e,\quad  F=-\nabla \Phi,\\
    \N=&\,\,n-\bar{n}, \quad \U=u-\bar{u},\quad \F= \nabla\phi-\nabla\bar{\phi}.
\end{align*}

\subsection*{Proof of Theorem \ref{theorem5.1}}
We rewrite the Euler-Poisson system \eqref{EPmain} as
\begin{equation*}
	\left\{\begin{array}{l}
		\partial_{t} N+\operatorname{div}((N+\bar{n}) u)=0, \\
		\eps^2\partial_{t} u+\eps^2(u \cdot \nabla) u+\nabla(h(N+\bar{n})-h(\bar{n}))+u=-\nabla \Phi, \\
		\Delta \Phi=-N,
	\end{array}\right.
\end{equation*}
in which the Euler equations are the special case of \eqref{EM30} with $B = 0$, while the Maxwell equations in \eqref{EM30} are replaced by $\Delta \Phi=-N$.

Now we want to establish an analogous energy estimate in the present case with $G = 0$. By checking all the steps, we see that the Maxwell equations are concerned only in the proof of Lemma \ref{lemma3.1} and Lemma \ref{lemma3.2}. Essentially, one has to deal with the quadratic term $2\left< \pt^k \pa_x^\al(nu),\pt^k F_\al\right>$ with $k+|\al|\le s$, appeared in the proof due to the Poisson equations. In our case, this term can be estimated as follows. Since
\begin{equation*}
	F=-\nabla \Phi, \quad \partial_{t} N=-\operatorname{div}(n u) \quad \text { and } \quad \Delta \Phi=-N,
\end{equation*}
one has by energy estimates,
\begin{align*}
	\left< \pt^k \pa_x^\al(nu),\pt^k F_\al\right>=&-\left< \pt^k \pa_x^\al(nu),\pt^k\pa_x^\al \nabla\Phi\right>
	=\left<\pt^k\pa_x^\al\dive{(nu)},\pt^k\pa_x^\al\Phi \right>\\
	=&-\left<\pt^{k+1}N_\al,\pt^k\pa_x^\al\Phi\right>
	=\left<\pt^{k+1}\Delta \pa_x^\al \Phi,\pt^k\pa_x^\al\Phi\right>=-\frac{\mathrm{d}}{\mathrm{d}t}\frac{1}{2}\|\pt^k F_\al\|^2.
\end{align*}
This shows the validity of all the steps before \eqref{middle4} and \eqref{N-s}, which imply \eqref{5.1.1}. The initial data of $F$ can be obtained through the Poisson equation
\[
\Delta \Phi(0,x)=-(n_0^\eps-n_e), \quad m_{\Phi}(t)=0.
\]
Similar to Theorem \ref{theorem2}, one obtains Theorem \ref{theorem5.1}. \hfill $\square$

\subsection*{Proof of Theorem \ref{theorem5.2}}
From \eqref{EPmain} and \eqref{EM-drift-diffusion}, one has
\begin{equation*}
	\dive{(\pt\F)}=-\pt{\N}=\dive{(nu-\bar{n}\bar{u})}.
\end{equation*}
Consequently, there exists a function $M$ such that the stream function $\F$ satisfies
\begin{equation}\label{EPstream}
	\pt{\F}=(nu-\bar{n}\bar{u})+\nabla\times M.
\end{equation}
Due to the similar structure of the Euler equations, by checking all steps, we find that the Maxwell equations are concerned only in Lemma \ref{lemma4.4}. More precisely, we just need to estimate the quadratic term $\left<\pa_x^\al(nu-\bar{n}\bar{u}),\F_\al\right>$ for $|\al|\leq s-1$. Indeed, by \eqref{EPstream} and the fact that $\F$ is rotation free, one obtains
\[
\left<\partial_x^\al(nu-\bar{n}\bar{u}),\F_\al \right>=\left<\partial_x^\al(\pt\F-\nabla\times M),\F_\al \right>=\frac{1}{2}\frac{\mathrm{d}}{\mathrm{d}t}\|\F_\al\|^2.
\]
The initial data of $\F$ can be obtained through the Poisson equation
\[
\Delta (\phi(0,x)-\bar{\phi}(0,x))=-(n_0^\eps-\bar{n}(0,x)), \quad m_{\Phi}(t)=0, \,\,\text{for}\,\,x\in\T^3.
\]
This shows the validity of all steps before \eqref{middle11} and thus one obtains \eqref{EPerror}. \hfill $\square$

\vspace{5mm}

\noindent{\bf Acknowledgments:} The research of this work was supported in part by the National Natural Science Foundation of China under grants 11831011 and 12161141004. This work was also partially supported by  the Fundamental Research Funds for the Central Universities and Shanghai Frontiers Science Center of Modern Analysis.

\vspace{5mm}

\noindent {\bf Conflict of Interest:} The authors declare that they have no conflict of interest.

\vspace{5mm}
\noindent {\bf Data availability:} Data sharing is  not applicable to this article as no data sets were generated or analysed during the current study.


\end{document}